\documentclass[proceedings,submission]{dmtcs}

\usepackage[latin1]{inputenc}
\usepackage{amsmath, amsfonts, amssymb}
\usepackage{enumerate}
\usepackage{subfigure}
\usepackage{tikz}
\usetikzlibrary[fit]

\author{Samuele Giraudo\addressmark{1}}
\title{Balanced binary trees in the Tamari lattice}
\address{\addressmark{1}Institut Gaspard Monge, Universit\'e Paris-Est
Marne-la-Vall\'ee, 5 Boulevard Descartes, Champs-sur-Marne,
77454 Marne-la-Vall\'ee cedex 2, France}

\newtheorem{theorem}{Theorem}[section]
\newtheorem{proposition}[theorem]{Proposition}
\newtheorem{lemma}[theorem]{Lemma}

\newtheorem{definition}[theorem]{Definition}
\newtheorem{remark}[theorem]{Remark}

\numberwithin{equation}{section}

\newcommand{\Ht}{\operatorname{ht}}
\newcommand{\EnsNat}{\mathbb{N}}
\newcommand{\EnsRel}{\mathbb{Z}}
\newcommand{\EnsArb}{\mathcal{T}}
\newcommand{\ArbCons}{\wedge}
\newcommand{\EnsEq}{\mathcal{B}}
\newcommand{\TamPar}{\rightthreetimes}
\newcommand{\TamOrd}{\preccurlyeq}
\newcommand{\Tam}{\mathbb{T}}
\newcommand{\Des}{\gamma}
\newcommand{\HyperCube}{\mathbb{H}}
\newcommand{\Potentiel}{\operatorname{P}}
\newcommand{\PropDes}{\operatorname{Imb}}
\newcommand{\ADroite}{\rightsquigarrow}
\newcommand{\EnsAdmissible}{\mathcal{A}}
\newcommand{\MotCar}{\operatorname{c}}

\tikzstyle{Noeud} = [circle, draw = blue!100, fill = blue!30, thick, inner sep = 0pt, minimum size = 6mm]
\tikzstyle{Feuille} = [rectangle, draw = black!100, fill = black!30, thick, inner sep = 0pt, minimum size = 2mm]
\tikzstyle{Bourgeon} = [circle, draw = green!100, fill = green!30, thick, inner sep = 0pt, minimum size = 6mm]
\tikzstyle{SArbre} = [rectangle, draw = orange!100, fill = orange!30, thick, inner sep = 0pt, minimum size = 8mm]
\tikzstyle{NoeudR} = [rectangle, draw = blue!100, fill = blue!30, thick, inner sep = 0pt, minimum size = 6mm]
\tikzstyle{Cercle} = [circle, draw = red!100, fill = red!70, thick, inner sep = 0pt, minimum size = 3mm]
\tikzstyle{ArreteCube} = [line width=1pt, draw=blue]
\tikzstyle{Arrete} = [red!70, thick, draw, line width = 2pt]
\tikzstyle{edge from parent} = [red!70, thick, draw, line width = 2pt]
\tikzstyle{level} = [level distance = 10mm]
\tikzstyle{level 1} = [sibling distance=24mm]
\tikzstyle{level 2} = [sibling distance=12mm]
\tikzstyle{level 3} = [sibling distance=6mm]

\newcommand{\Feuille}[1]{%
    \scalebox{#1}{%
    \begin{tikzpicture}
        \node[Feuille]{};
    \end{tikzpicture}%
    }%
}

\newcommand{\Noeud}[2]{%
    \scalebox{#1}{%
        \begin{tikzpicture}%
            \node[Noeud]{\Large #2};%
        \end{tikzpicture}%
    }%
}

\newcommand{\Bourgeon}[2]{%
    \scalebox{#1}{%
        \begin{tikzpicture}%
            \node[Bourgeon]{\Large #2};%
        \end{tikzpicture}%
    }%
}

\newcommand{\BourgeonA}[4]{%
    \scalebox{#1}{%
    \begin{tikzpicture}%
        \node[Bourgeon](1)at(0,0){\Large  #2};%
        \node[Noeud](2)at(1,1){\Large  #3};%
        \node[Bourgeon](3)at(2,0){\Large  #4};%
        \draw[Arrete](1)--(2);%
        \draw[Arrete](2)--(3);%
    \end{tikzpicture}%
    }%
}

\newcommand{\BourgeonB}[4]{%
    \scalebox{#1}{%
    \begin{tikzpicture}%
        \tikzstyle{NoeudR} = [rectangle, draw = blue!100, fill = blue!30, thick, inner sep = 0pt, minimum size = 6mm]%
        \node[Bourgeon](0)at(0,-2){\Large $x$};%
        \node[Noeud](1)at(1,-1){\Large #2};%
        \draw[Arrete](1)--(0);%
        \node[Bourgeon](2)at(2,-2){\Large #3};%
        \draw[Arrete](1)--(2);%
        \node[NoeudR](3)at(3,0){\Large #4};%
        \draw[Arrete](3)--(1);%
        \node[Bourgeon](4)at(4,-1){\Large $x$};%
        \draw[Arrete](3)--(4);%
    \end{tikzpicture}%
    }%
}

\newcommand{\MotifA}[3]{%
    \scalebox{#1}{%
    \begin{tikzpicture}%
        \tikzstyle{Noeud} = [circle, draw = blue!100, fill = blue!30, thick, inner sep = 0pt, minimum size = 6mm]
        \node[Noeud](0)at(0,-1){\large #2};%
        \node[Noeud](1)at(1,0){\large #3};%
        \draw[Arrete](1)--(0);%
    \end{tikzpicture}%
    }%
}

\newcommand{\MotifB}[3]{%
    \scalebox{#1}{%
    \begin{tikzpicture}%
        \tikzstyle{Noeud} = [circle, draw = blue!100, fill = blue!30, thick, inner sep = 0pt, minimum size = 6mm]
        \node[Noeud](0)at(0,0){\large #2};%
        \node[Noeud](1)at(1,-1){\large #3};%
        \draw[Arrete](0)--(1);
    \end{tikzpicture}%
    }%
}

\newcommand{\ArbreA}[1]{%
    \scalebox{#1}{%
    \begin{tikzpicture}%
        \node[Feuille](0)at(0,-1){};%
        \node[Noeud](1)at(1,0){};%
        \node[Feuille](2)at(2,-1){};%
        \draw[Arrete](1)--(0);%
        \draw[Arrete](1)--(2);%
    \end{tikzpicture}%
    }%
}

\newcommand{\ArbreB}[1]{%
    \scalebox{#1}{%
    \begin{tikzpicture}%
        \node[Feuille](0)at(0,-2){};%
        \node[Noeud](1)at(1,-1){};%
        \node[Feuille](2)at(2,-2){};%
        \draw[Arrete](1)--(0);%
        \draw[Arrete](1)--(2);%
        \node[Noeud](3)at(3,0){};%
        \node[Feuille](4)at(4,-1){};%
        \draw[Arrete](3)--(1);%
        \draw[Arrete](3)--(4);%
    \end{tikzpicture}%
    }%
}

\begin{document}

\maketitle

\begin{abstract}
    $\newline$
    \paragraph{Abstract.}
    We show that the set of balanced binary trees is closed by interval
    in the Tamari lattice. We establish that the intervals $[T_0, T_1]$
    where $T_0$ and $T_1$ are balanced trees are isomorphic as posets to
    a hypercube. We introduce tree patterns and synchronous grammars to
    get a functional equation of the generating series enumerating balanced
    tree intervals.

    \paragraph{R\'esum\'e.}
    Nous montrons que l'ensemble des arbres \'equilibr\'es est clos par
    intervalle dans le treillis de Tamari. Nous caract\'erisons la forme
    des intervalles du type $[T_0, T_1]$ o\`u $T_0$ et $T_1$ sont \'equilibr\'es
    en montrant qu'en tant qu'ensembles partiellement ordonn\'es, ils sont
    isomorphes \`a un hypercube. Nous introduisons la notion de motif
    d'arbre et de grammaire synchrone dans le but d'\'etablir une
    \'equation fonctionnelle de la s\'erie g\'en\'eratrice qui d\'enombre
    les intervalles d'arbres \'equilibr\'es.
\end{abstract}

\section{Introduction}

Binary search trees are used as data structures to represent dynamic totally ordered
sets~\cite{KNU398, CLR04, AU93}. The algorithms solving classical related
problems such as the insertion, the deletion or the search of a given
element can be performed in a time logarithmic in the cardinality of the
represented set, provided that the encoding binary tree is balanced. Recall
that a binary tree is balanced if for each node $x$, the height of the
left subtree of $x$ and the height of the right subtree of $x$ differ by
at most one.

The algorithmic of balanced trees relies fundamentally on the so-called
rotation operation. An insertion or a deletion of an element in a dynamic
ordered set modifies the tree encoding it and can imbalance it. The
efficiency of these algorithms comes from the fact that binary search
trees can be rebalanced very quickly after the insertion or the deletion,
using no more than two rotations~\cite{AVL62}.

Surprisingly, this operation appears in a different context since it defines
a partial order on the set of binary trees of a given size. A tree $T_0$
is smaller than a tree $T_1$ if it is possible to transform the tree $T_0$
into the tree $T_1$ by performing a succession of right rotations. This
partial order, known as the Tamari order~\cite{KNU44, STA99}, defines a
lattice structure on the set of binary trees of a given size.

Since binary trees are naturally equipped with this order structure induced
by rotations, and the balance of balanced trees is maintained doing rotations,
we would like to investigate if balanced trees play a particular role in
the Tamari lattice. Our goal, in this is paper, is to combine the two
points of view of the rotation operation. A first simple computer observation is
that the intervals $[T_0,~T_1]$ where $T_0$ and $T_1$ are balanced trees
are only made up of balanced trees. The main goal of this paper is to prove
this property. As a consequence, we give a characterization on the shape of
these intervals and, using grammars allowing to generate trees, enumerate them.

This article is organized as follows. In Section \ref{Sec:Preliminaries},
we set the essential notions about binary trees and balanced trees, and
we give the definition of the Tamari lattice in our setting. Section
\ref{Sec:Closure} is devoted to establish the main result: the set of
balanced trees is closed by interval in the Tamari lattice. In Section
\ref{Sec:Grammars}, we define tree patterns and synchronous grammars. These
grammars allow us to generate trees avoiding a given set of tree patterns.
We define a subset of balanced trees where elements hold a peculiar position
in the Tamari lattice and we give, using the synchronous grammar generating
these, a functional equation of the generating series enumerating these.
Finally, in Section \ref{Sec:Shape}, we look at balanced tree intervals
and show that they are, as posets, isomorphic to hypercubes. Encoding balanced
tree intervals by particular trees, and establishing the synchronous grammar
generating these trees, we give a functional equation satisfied by the
generating series enumerating balanced tree intervals.

\section*{Acknowledgments}
The author would like to thank Florent Hivert for introducing him to the
problem addressed in this paper, and Jean-Christophe Novelli and Florent
Hivert for their invaluable advice and their improvement suggestions. The
computations of this work have been done with the open-source mathematical
software Sage~\cite{SAGE}.

\section{Preliminaries} \label{Sec:Preliminaries}

\subsection{Complete rooted planar binary trees}

In this article, we consider complete rooted planar binary trees. Nodes
are denoted by circles like \Noeud{.4}{} and leaves by
squares like \raisebox{0.2em}{\Feuille{.4}}. The empty tree is also denoted by \raisebox{0.2em}{\Feuille{.4}}.
Assuming $L$ and $R$ are complete rooted planar binary trees, let $L \ArbCons R$
be the (unique) complete rooted planar binary tree which has $L$ as left
subtree and $R$ as right subtree. Let also $\EnsArb_n$ be the set of complete
rooted planar binary trees with $n$ nodes and $\EnsArb$ be the set of all
complete rooted planar binary trees. We use in the sequel the standard
terminology (ie. \emph{child}, \emph{ancestor}, \emph{edge}, \emph{path},
\ldots) about complete rooted planar binary trees ~\cite{AU93}.

Recall that the nodes of a complete rooted planar binary tree $T$ can be
visited in the infix order: it consists in visiting recursively the left
subtree of $T$, then the root, and finally the right subtree. We say that
a node $y$ is \emph{on the right compared to} a node $x$ in $T$
if the node $x$ appears strictly before the node $y$ in the infix order
and we denote that by $x \ADroite_T y$. We extend this notation to subtrees saying that
a subtree $S$ of root $y$ of $T$ is on the right compared to a node $x$
in $T$ if for all nodes $y'$ of $S$ we have $x \ADroite_T y'$. We say that
a node $x$ of $T$ is the \emph{leftmost node} of $T$ if $x$ is the first
visited node in the infix order.

If $T$ is a complete rooted planar binary tree, we shall denote by $\Ht(T)$
the \emph{height} of $T$, that is the length of the longest path connecting
the root of $T$ to one of its leaves. For example, we have $\Ht \left( \raisebox{0.2em}{\Feuille{.4}} \right) = 0$,
$\Ht \left( \ArbreA{.17} \right) = 1$, and $\Ht \left( \ArbreB{.17} \right) = 2$.

In the sequel, we shall mainly talk about complete rooted planar binary
trees so we shall call them simply \emph{trees}.

\subsection{Balanced trees}

Let us define, for each tree $T$, the mapping $\Des_T$ called the \emph{imbalance
mapping} which associates an element of $\EnsRel$ with a node $x$ of $T$,
namely the \emph{imbalance value} of $x$. It is defined for a node $x$ by:
\begin{equation}
    \Des_T(x) = \Ht(R) - \Ht(L)
\end{equation}
where $L$ (resp. $R$) is the left (resp. right) subtree of $x$.

Balanced trees form a subset of $\EnsArb$ composed of trees which have
the property of being balanced:

\begin{definition}
    A tree $T$ is \emph{balanced} if for all node $x$ of $T$, we have
    \begin{equation}
        \Des_T(x) \in \{-1, 0, 1\}.
    \end{equation}
\end{definition}

Let us denote by $\EnsEq_n$ the set of balanced trees with $n$ nodes (see
Figure \ref{FIGarbresEq} for the first sets) and $\EnsEq$ the set of all
balanced trees.

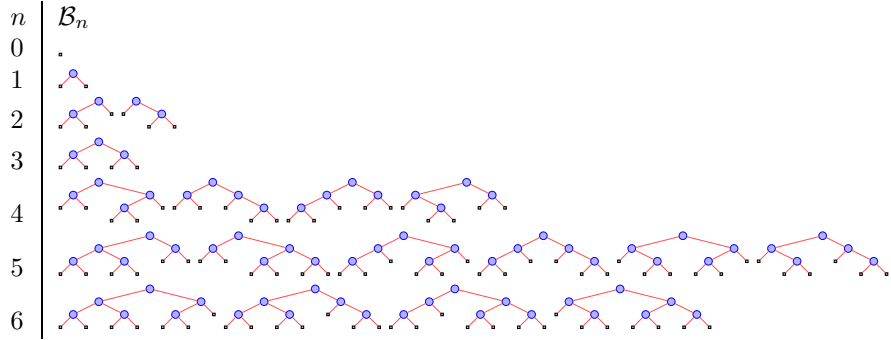
\begin{figure}[ht]
    \centering
        \begin{tabular}{l|l}
            $n$ & $\EnsEq_n$ \\
            $0$ & \scalebox{.17}{
                \begin{tikzpicture}
                    \node[Feuille]{};
                \end{tikzpicture}
            } \\
            $1$ & \scalebox{.17}{
                \begin{tikzpicture}
                    \node[Feuille](0)at(0,-1){};
                    \node[Noeud](1)at(1,0){};
                    \node[Feuille](2)at(2,-1){};
                    \draw[Arrete](1)--(0);
                    \draw[Arrete](1)--(2);
                \end{tikzpicture}
            } \\
            $2$ & \scalebox{.17}{
                \begin{tikzpicture}
                    \node[Feuille](0)at(0,-2){};
                    \node[Noeud](1)at(1,-1){};
                    \node[Feuille](2)at(2,-2){};
                    \draw[Arrete](1)--(0);
                    \draw[Arrete](1)--(2);
                    \node[Noeud](3)at(3,0){};
                    \node[Feuille](4)at(4,-1){};
                    \draw[Arrete](3)--(1);
                    \draw[Arrete](3)--(4);
                \end{tikzpicture}
            }
                  \scalebox{.17}{
                \begin{tikzpicture}
                    \node[Feuille](0)at(0,-1){};
                    \node[Noeud](1)at(1,0){};
                    \node[Feuille](2)at(2,-2){};
                    \node[Noeud](3)at(3,-1){};
                    \node[Feuille](4)at(4,-2){};
                    \draw[Arrete](3)--(2);
                    \draw[Arrete](3)--(4);
                    \draw[Arrete](1)--(0);
                    \draw[Arrete](1)--(3);
                \end{tikzpicture}
            } \\
            $3$ & \scalebox{.17}{
                \begin{tikzpicture}
                    \node[Feuille](0)at(0,-2){};
                    \node[Noeud](1)at(1,-1){};
                    \node[Feuille](2)at(2,-2){};
                    \draw[Arrete](1)--(0);
                    \draw[Arrete](1)--(2);
                    \node[Noeud](3)at(3,0){};
                    \node[Feuille](4)at(4,-2){};
                    \node[Noeud](5)at(5,-1){};
                    \node[Feuille](6)at(6,-2){};
                    \draw[Arrete](5)--(4);
                    \draw[Arrete](5)--(6);
                    \draw[Arrete](3)--(1);
                    \draw[Arrete](3)--(5);
                \end{tikzpicture}
            } \\
            $4$ & \scalebox{.17}{
                \begin{tikzpicture}
                    \node[Feuille](0)at(0,-2){};
                    \node[Noeud](1)at(1,-1){};
                    \node[Feuille](2)at(2,-2){};
                    \draw[Arrete](1)--(0);
                    \draw[Arrete](1)--(2);
                    \node[Noeud](3)at(3,0){};
                    \node[Feuille](4)at(4,-3){};
                    \node[Noeud](5)at(5,-2){};
                    \node[Feuille](6)at(6,-3){};
                    \draw[Arrete](5)--(4);
                    \draw[Arrete](5)--(6);
                    \node[Noeud](7)at(7,-1){};
                    \node[Feuille](8)at(8,-2){};
                    \draw[Arrete](7)--(5);
                    \draw[Arrete](7)--(8);
                    \draw[Arrete](3)--(1);
                    \draw[Arrete](3)--(7);
                \end{tikzpicture}
            }
                  \scalebox{.17}{
                \begin{tikzpicture}
                    \node[Feuille](0)at(0,-2){};
                    \node[Noeud](1)at(1,-1){};
                    \node[Feuille](2)at(2,-2){};
                    \draw[Arrete](1)--(0);
                    \draw[Arrete](1)--(2);
                    \node[Noeud](3)at(3,0){};
                    \node[Feuille](4)at(4,-2){};
                    \node[Noeud](5)at(5,-1){};
                    \node[Feuille](6)at(6,-3){};
                    \node[Noeud](7)at(7,-2){};
                    \node[Feuille](8)at(8,-3){};
                    \draw[Arrete](7)--(6);
                    \draw[Arrete](7)--(8);
                    \draw[Arrete](5)--(4);
                    \draw[Arrete](5)--(7);
                    \draw[Arrete](3)--(1);
                    \draw[Arrete](3)--(5);
                \end{tikzpicture}
                  }
                  \scalebox{.17}{
                \begin{tikzpicture}
                    \node[Feuille](0)at(0,-3){};
                    \node[Noeud](1)at(1,-2){};
                    \node[Feuille](2)at(2,-3){};
                    \draw[Arrete](1)--(0);
                    \draw[Arrete](1)--(2);
                    \node[Noeud](3)at(3,-1){};
                    \node[Feuille](4)at(4,-2){};
                    \draw[Arrete](3)--(1);
                    \draw[Arrete](3)--(4);
                    \node[Noeud](5)at(5,0){};
                    \node[Feuille](6)at(6,-2){};
                    \node[Noeud](7)at(7,-1){};
                    \node[Feuille](8)at(8,-2){};
                    \draw[Arrete](7)--(6);
                    \draw[Arrete](7)--(8);
                    \draw[Arrete](5)--(3);
                    \draw[Arrete](5)--(7);
                \end{tikzpicture}
                  }
                  \scalebox{.17}{
                  \begin{tikzpicture}
                    \node[Feuille](0)at(0,-2){};
                    \node[Noeud](1)at(1,-1){};
                    \node[Feuille](2)at(2,-3){};
                    \node[Noeud](3)at(3,-2){};
                    \node[Feuille](4)at(4,-3){};
                    \draw[Arrete](3)--(2);
                    \draw[Arrete](3)--(4);
                    \draw[Arrete](1)--(0);
                    \draw[Arrete](1)--(3);
                    \node[Noeud](5)at(5,0){};
                    \node[Feuille](6)at(6,-2){};
                    \node[Noeud](7)at(7,-1){};
                    \node[Feuille](8)at(8,-2){};
                    \draw[Arrete](7)--(6);
                    \draw[Arrete](7)--(8);
                    \draw[Arrete](5)--(1);
                    \draw[Arrete](5)--(7);
                \end{tikzpicture}
                  } \\
            $5$ & \scalebox{.17}{
                \begin{tikzpicture}
                    \node[Feuille](0)at(0,-3){};
                    \node[Noeud](1)at(1,-2){};
                    \node[Feuille](2)at(2,-3){};
                    \draw[Arrete](1)--(0);
                    \draw[Arrete](1)--(2);
                    \node[Noeud](3)at(3,-1){};
                    \node[Feuille](4)at(4,-3){};
                    \node[Noeud](5)at(5,-2){};
                    \node[Feuille](6)at(6,-3){};
                    \draw[Arrete](5)--(4);
                    \draw[Arrete](5)--(6);
                    \draw[Arrete](3)--(1);
                    \draw[Arrete](3)--(5);
                    \node[Noeud](7)at(7,0){};
                    \node[Feuille](8)at(8,-2){};
                    \node[Noeud](9)at(9,-1){};
                    \node[Feuille](10)at(10,-2){};
                    \draw[Arrete](9)--(8);
                    \draw[Arrete](9)--(10);
                    \draw[Arrete](7)--(3);
                    \draw[Arrete](7)--(9);
                \end{tikzpicture}
                }
                  \scalebox{.17}{
                  \begin{tikzpicture}
                    \node[Feuille](0)at(0,-2){};
                    \node[Noeud](1)at(1,-1){};
                    \node[Feuille](2)at(2,-2){};
                    \draw[Arrete](1)--(0);
                    \draw[Arrete](1)--(2);
                    \node[Noeud](3)at(3,0){};
                    \node[Feuille](4)at(4,-3){};
                    \node[Noeud](5)at(5,-2){};
                    \node[Feuille](6)at(6,-3){};
                    \draw[Arrete](5)--(4);
                    \draw[Arrete](5)--(6);
                    \node[Noeud](7)at(7,-1){};
                    \node[Feuille](8)at(8,-3){};
                    \node[Noeud](9)at(9,-2){};
                    \node[Feuille](10)at(10,-3){};
                    \draw[Arrete](9)--(8);
                    \draw[Arrete](9)--(10);
                    \draw[Arrete](7)--(5);
                    \draw[Arrete](7)--(9);
                    \draw[Arrete](3)--(1);
                    \draw[Arrete](3)--(7);
                \end{tikzpicture}
                  }
                  \scalebox{.17}{
                  \begin{tikzpicture}
                    \node[Feuille](0)at(0,-3){};
                    \node[Noeud](1)at(1,-2){};
                    \node[Feuille](2)at(2,-3){};
                    \draw[Arrete](1)--(0);
                    \draw[Arrete](1)--(2);
                    \node[Noeud](3)at(3,-1){};
                    \node[Feuille](4)at(4,-2){};
                    \draw[Arrete](3)--(1);
                    \draw[Arrete](3)--(4);
                    \node[Noeud](5)at(5,0){};
                    \node[Feuille](6)at(6,-3){};
                    \node[Noeud](7)at(7,-2){};
                    \node[Feuille](8)at(8,-3){};
                    \draw[Arrete](7)--(6);
                    \draw[Arrete](7)--(8);
                    \node[Noeud](9)at(9,-1){};
                    \node[Feuille](10)at(10,-2){};
                    \draw[Arrete](9)--(7);
                    \draw[Arrete](9)--(10);
                    \draw[Arrete](5)--(3);
                    \draw[Arrete](5)--(9);
                \end{tikzpicture}
                  }
                  \scalebox{.17}{
                \begin{tikzpicture}
                    \node[Feuille](0)at(0,-3){};
                    \node[Noeud](1)at(1,-2){};
                    \node[Feuille](2)at(2,-3){};
                    \draw[Arrete](1)--(0);
                    \draw[Arrete](1)--(2);
                    \node[Noeud](3)at(3,-1){};
                    \node[Feuille](4)at(4,-2){};
                    \draw[Arrete](3)--(1);
                    \draw[Arrete](3)--(4);
                    \node[Noeud](5)at(5,0){};
                    \node[Feuille](6)at(6,-2){};
                    \node[Noeud](7)at(7,-1){};
                    \node[Feuille](8)at(8,-3){};
                    \node[Noeud](9)at(9,-2){};
                    \node[Feuille](10)at(10,-3){};
                    \draw[Arrete](9)--(8);
                    \draw[Arrete](9)--(10);
                    \draw[Arrete](7)--(6);
                    \draw[Arrete](7)--(9);
                    \draw[Arrete](5)--(3);
                    \draw[Arrete](5)--(7);
                \end{tikzpicture}
                  }
                  \scalebox{.17}{
                \begin{tikzpicture}
                    \node[Feuille](0)at(0,-2){};
                    \node[Noeud](1)at(1,-1){};
                    \node[Feuille](2)at(2,-3){};
                    \node[Noeud](3)at(3,-2){};
                    \node[Feuille](4)at(4,-3){};
                    \draw[Arrete](3)--(2);
                    \draw[Arrete](3)--(4);
                    \draw[Arrete](1)--(0);
                    \draw[Arrete](1)--(3);
                    \node[Noeud](5)at(5,0){};
                    \node[Feuille](6)at(6,-3){};
                    \node[Noeud](7)at(7,-2){};
                    \node[Feuille](8)at(8,-3){};
                    \draw[Arrete](7)--(6);
                    \draw[Arrete](7)--(8);
                    \node[Noeud](9)at(9,-1){};
                    \node[Feuille](10)at(10,-2){};
                    \draw[Arrete](9)--(7);
                    \draw[Arrete](9)--(10);
                    \draw[Arrete](5)--(1);
                    \draw[Arrete](5)--(9);
                \end{tikzpicture}
                  }
                  \scalebox{.17}{
                  \begin{tikzpicture}
                    \node[Feuille](0)at(0,-2){};
                    \node[Noeud](1)at(1,-1){};
                    \node[Feuille](2)at(2,-3){};
                    \node[Noeud](3)at(3,-2){};
                    \node[Feuille](4)at(4,-3){};
                    \draw[Arrete](3)--(2);
                    \draw[Arrete](3)--(4);
                    \draw[Arrete](1)--(0);
                    \draw[Arrete](1)--(3);
                    \node[Noeud](5)at(5,0){};
                    \node[Feuille](6)at(6,-2){};
                    \node[Noeud](7)at(7,-1){};
                    \node[Feuille](8)at(8,-3){};
                    \node[Noeud](9)at(9,-2){};
                    \node[Feuille](10)at(10,-3){};
                    \draw[Arrete](9)--(8);
                    \draw[Arrete](9)--(10);
                    \draw[Arrete](7)--(6);
                    \draw[Arrete](7)--(9);
                    \draw[Arrete](5)--(1);
                    \draw[Arrete](5)--(7);
                \end{tikzpicture}
                  } \\
            $6$ & \scalebox{.17}{
                \begin{tikzpicture}
                    \node[Feuille](0)at(0,-3){};
                    \node[Noeud](1)at(1,-2){};
                    \node[Feuille](2)at(2,-3){};
                    \draw[Arrete](1)--(0);
                    \draw[Arrete](1)--(2);
                    \node[Noeud](3)at(3,-1){};
                    \node[Feuille](4)at(4,-3){};
                    \node[Noeud](5)at(5,-2){};
                    \node[Feuille](6)at(6,-3){};
                    \draw[Arrete](5)--(4);
                    \draw[Arrete](5)--(6);
                    \draw[Arrete](3)--(1);
                    \draw[Arrete](3)--(5);
                    \node[Noeud](7)at(7,0){};
                    \node[Feuille](8)at(8,-3){};
                    \node[Noeud](9)at(9,-2){};
                    \node[Feuille](10)at(10,-3){};
                    \draw[Arrete](9)--(8);
                    \draw[Arrete](9)--(10);
                    \node[Noeud](11)at(11,-1){};
                    \node[Feuille](12)at(12,-2){};
                    \draw[Arrete](11)--(9);
                    \draw[Arrete](11)--(12);
                    \draw[Arrete](7)--(3);
                    \draw[Arrete](7)--(11);
                \end{tikzpicture}
                }
                  \scalebox{.17}{
                    \begin{tikzpicture}
                        \node[Feuille](0)at(0,-3){};
                        \node[Noeud](1)at(1,-2){};
                        \node[Feuille](2)at(2,-3){};
                        \draw[Arrete](1)--(0);
                        \draw[Arrete](1)--(2);
                        \node[Noeud](3)at(3,-1){};
                        \node[Feuille](4)at(4,-3){};
                        \node[Noeud](5)at(5,-2){};
                        \node[Feuille](6)at(6,-3){};
                        \draw[Arrete](5)--(4);
                        \draw[Arrete](5)--(6);
                        \draw[Arrete](3)--(1);
                        \draw[Arrete](3)--(5);
                        \node[Noeud](7)at(7,0){};
                        \node[Feuille](8)at(8,-2){};
                        \node[Noeud](9)at(9,-1){};
                        \node[Feuille](10)at(10,-3){};
                        \node[Noeud](11)at(11,-2){};
                        \node[Feuille](12)at(12,-3){};
                        \draw[Arrete](11)--(10);
                        \draw[Arrete](11)--(12);
                        \draw[Arrete](9)--(8);
                        \draw[Arrete](9)--(11);
                        \draw[Arrete](7)--(3);
                        \draw[Arrete](7)--(9);
                    \end{tikzpicture}
                  }
                  \scalebox{.17}{
                    \begin{tikzpicture}
                        \node[Feuille](0)at(0,-3){};
                        \node[Noeud](1)at(1,-2){};
                        \node[Feuille](2)at(2,-3){};
                        \draw[Arrete](1)--(0);
                        \draw[Arrete](1)--(2);
                        \node[Noeud](3)at(3,-1){};
                        \node[Feuille](4)at(4,-2){};
                        \draw[Arrete](3)--(1);
                        \draw[Arrete](3)--(4);
                        \node[Noeud](5)at(5,0){};
                        \node[Feuille](6)at(6,-3){};
                        \node[Noeud](7)at(7,-2){};
                        \node[Feuille](8)at(8,-3){};
                        \draw[Arrete](7)--(6);
                        \draw[Arrete](7)--(8);
                        \node[Noeud](9)at(9,-1){};
                        \node[Feuille](10)at(10,-3){};
                        \node[Noeud](11)at(11,-2){};
                        \node[Feuille](12)at(12,-3){};
                        \draw[Arrete](11)--(10);
                        \draw[Arrete](11)--(12);
                        \draw[Arrete](9)--(7);
                        \draw[Arrete](9)--(11);
                        \draw[Arrete](5)--(3);
                        \draw[Arrete](5)--(9);
                    \end{tikzpicture}
                  }
                  \scalebox{.17}{
                  \begin{tikzpicture}
                    \node[Feuille](0)at(0,-2){};
                    \node[Noeud](1)at(1,-1){};
                    \node[Feuille](2)at(2,-3){};
                    \node[Noeud](3)at(3,-2){};
                    \node[Feuille](4)at(4,-3){};
                    \draw[Arrete](3)--(2);
                    \draw[Arrete](3)--(4);
                    \draw[Arrete](1)--(0);
                    \draw[Arrete](1)--(3);
                    \node[Noeud](5)at(5,0){};
                    \node[Feuille](6)at(6,-3){};
                    \node[Noeud](7)at(7,-2){};
                    \node[Feuille](8)at(8,-3){};
                    \draw[Arrete](7)--(6);
                    \draw[Arrete](7)--(8);
                    \node[Noeud](9)at(9,-1){};
                    \node[Feuille](10)at(10,-3){};
                    \node[Noeud](11)at(11,-2){};
                    \node[Feuille](12)at(12,-3){};
                    \draw[Arrete](11)--(10);
                    \draw[Arrete](11)--(12);
                    \draw[Arrete](9)--(7);
                    \draw[Arrete](9)--(11);
                    \draw[Arrete](5)--(1);
                    \draw[Arrete](5)--(9);
                \end{tikzpicture}
                  }
        \end{tabular}
    \caption{The first balanced trees.}
    \label{FIGarbresEq}
\end{figure}

\subsection{The Tamari lattice}

The Tamari lattice can be defined in several ways~\cite{STA99, BEB07}
depending on which kind of catalan object (ie. in bijection with trees)
the order relation is defined. We give here the most convenient definition
for our use. First, let us recall the right rotation operation:

\begin{definition}
    Let $T_0$ be a tree and $S_0 = (A \ArbCons B) \ArbCons C$ be the subtree
    of root $y$ of $T_0$. If $T_1$ is the tree obtained by replacing the
    tree $S_0$ by the tree $A \ArbCons (B \ArbCons C)$ in $T_0$ (see
    Figure \ref{FIGrotation}), we say that $T_1$ is obtained from $T_0$
    by a \emph{right rotation} of \emph{root} $y$.
\end{definition}

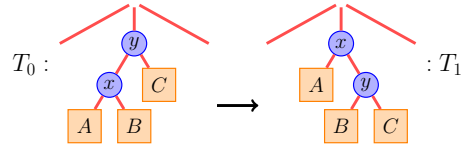
\begin{figure}[ht]
    \centering
        \scalebox{.55}{
            \begin{tikzpicture}
                \tikzstyle{level 1}=[sibling distance=18mm]
                \tikzstyle{level 2}=[sibling distance=12mm]
                \tikzstyle{level 3}=[sibling distance=12mm]
                \node[] {}
                    child[Noeud]
                    child{node[Noeud]{\Large $y$}
                        child{node[Noeud]{\Large $x$}
                            child{node[SArbre]{\Large $A$}
                                child[missing]
                                child[missing]
                            }
                            child{node[SArbre]{\Large $B$}
                                child[missing]
                                child[missing]
                            }
                        }
                        child{node[SArbre]{\Large $C$}
                            child[missing]
                            child[missing]
                        }
                    }
                    child[Noeud]
                ;
                \draw[line width=2pt, ->] (2,-2.5) -- (3,-2.5);
                \node[] at (5, 0) {}
                    child[Noeud]
                    child{node[Noeud] {\Large $x$}
                        child{node[SArbre]{\Large $A$}
                            child[missing]
                            child[missing]
                        }
                        child{node[Noeud]{\Large $y$}
                            child{node[SArbre]{\Large $B$}
                                child[missing]
                                child[missing]
                            }
                            child{node[SArbre]{\Large $C$}
                                child[missing]
                                child[missing]
                            }
                        }
                    }
                    child[Noeud]
                ;
                \node at (-2.5, -1.5) {\LARGE $T_0:$};
                \node at (7.5, -1.5) {\LARGE $:T_1$};
            \end{tikzpicture}
        }
    \caption{The right rotation of root $y$.}
    \label{FIGrotation}
\end{figure}

We write $T_0 \TamPar T_1$ if $T_1$ can be obtained by a right rotation
from $T_0$. We call the relation $\TamPar$ the \emph{partial Tamari relation}.

\begin{remark} \label{REMrotInf}
    Applying a right rotation to a tree does not change the infix order
    of its nodes.
\end{remark}

In the sequel, we only talk about right rotations, so we call these simply
\emph{rotations}. We are now in a position to give our definition of the
Tamari relation:

\begin{definition}
    The \emph{Tamari relation}, written $\TamOrd$, is the reflexive and
    transitive closure of the partial Tamari relation $\TamPar$.
\end{definition}

\begin{figure}[ht]
    \centering
    \subfigure[$\Tam_3$]{\makebox[5cm]{\scalebox{.14}{
        \begin{tikzpicture}
            \tikzstyle{Cercle} = [circle,draw=cyan,inner sep=0pt,thick,minimum size = 55mm, line width=4pt]
            \tikzstyle{Arc} = [line width=4pt, -, draw=blue]
            \node[Noeud] at (0, -1) (arbre111) {}
                child{node[Noeud] (arbre111n1) {}
                        child{node[Noeud] (arbre111n2) {}
                                child{node[Feuille] (arbre111f1) {}}
                                child{node[Feuille] (arbre111f2) {}}
                    }
                        child{node[Feuille] (arbre111f3) {}}
                }
                child{node[Feuille] (arbre111f4) {}};
            \node[Noeud] at (5, 8) (arbre121) {}
                child{node[Noeud] (arbre121n1) {}
                        child{node[Feuille] (arbre121f1) {}}
                        child{node[Feuille] (arbre121f2) {}}
                }
                child{node[Noeud] (arbre121n2) {}
                        child{node[Feuille] (arbre121f3) {}}
                        child{node[Feuille] (arbre121f4) {}}
                };
            \node[Noeud] at (-5, 4) (arbre211) {}
                child{node[Noeud] (arbre211n1) {}
                        child{node[Feuille] (arbre211f1) {}}
                        child{node[Noeud] (arbre211n2) {}
                                child{node[Feuille] (arbre211f2) {}}
                                child{node[Feuille] (arbre211f3) {}}
                    }
                }
                child{node[Feuille] (arbre211f4) {}};
            \node[Noeud] at (-5.55, 12) (arbre311) {}
                child{node[Feuille] (arbre311f1) {}}
                child{node[Noeud] (arbre311n1) {}
                        child{node[Noeud] (arbre311n2) {}
                                child{node[Feuille] (arbre311f2) {}}
                                child{node[Feuille] (arbre311f3) {}}
                    }
                        child{node[Feuille] (arbre311f4) {}}
                };
            \node[Noeud] at (0, 17) (arbre321) {}
                child{node[Feuille] (arbre321f1) {}}
                child{node[Noeud] (arbre321n1) {}
                        child{node[Feuille] (arbre321f2) {}}
                        child{node[Noeud] (arbre321n2) {}
                                child{node[Feuille] (arbre321f3) {}}
                                child{node[Feuille] (arbre321f4) {}}
                    }
                };
            \node[Cercle,fit=(arbre111) (arbre111n1) (arbre111n2) (arbre111f1) (arbre111f2) (arbre111f3) (arbre111f4)] (cercle111) {};
            \node[Cercle,fit=(arbre121) (arbre121n1) (arbre121n2) (arbre121f1) (arbre121f2) (arbre121f3) (arbre121f4)] (cercle121) {};
            \node[Cercle,fit=(arbre211) (arbre211n1) (arbre211n2) (arbre211f1) (arbre211f2) (arbre211f3) (arbre211f4)] (cercle211) {};
            \node[Cercle,fit=(arbre311) (arbre311n1) (arbre311n2) (arbre311f1) (arbre311f2) (arbre311f3) (arbre311f4)] (cercle311) {};
            \node[Cercle,fit=(arbre321) (arbre321n1) (arbre321n2) (arbre321f1) (arbre321f2) (arbre321f3) (arbre321f4)] (cercle321) {};
            \draw[Arc] (cercle111) -- (cercle121);
            \draw[Arc] (cercle111) -- (cercle211);
            \draw[Arc] (cercle121) -- (cercle321);
            \draw[Arc] (cercle211) -- (cercle311);
            \draw[Arc] (cercle311) -- (cercle321);
        \end{tikzpicture}
    }}}
    \subfigure[$\Tam_4$]{\makebox[7cm]{\scalebox{.14}{
        \begin{tikzpicture}
            \tikzstyle{Cercle} = [circle,draw=cyan,inner sep=0pt,thick,minimum size = 60mm, line width=4pt]
            \tikzstyle{Arc} = [line width=4pt, -, draw=blue]
            \node[Noeud] at (0, 0) (1111) {}
                child{node[Noeud] (11111) {}
                        child{node[Noeud](11112) {}
                                child{node[Noeud] (11113) {}
                                        child{node[Feuille] (11114) {}}
                                        child{node[Feuille] (11115) {}}
                        }
                                child{node[Feuille] (11116) {}}
                    }
                        child{node[Feuille] (11117) {}}
                }
                child{node[Feuille] (11118) {}};
            \node[Cercle,fit=(1111)(11111)(11112)(11113)(11114)(11115)(11116)(11117)] (c1111) {};
            \node[Noeud] at (-.75, 11) (1121) {}
                child{node[Noeud] (11211) {}
                        child{node[Noeud] (11212) {}
                                child{node[Feuille] (11213) {}}
                                child{node[Feuille] (11214) {}}
                    }
                        child{node[Feuille] (11215) {}}
                }
                child{node[Noeud] (11216) {}
                        child{node[Feuille] (11217) {}}
                        child{node[Feuille] (11218) {}}
                };
            \node[Cercle,fit=(1121)(11211)(11212)(11213)(11214)(11215)(11216)(11217)] (c1121) {};
            \draw[Arc] (c1111) -- (c1121);
            \node[Noeud] at (10, 6) (1211) {}
                child{node[Noeud] (12111) {}
                        child{node[Noeud] (12112) {}
                                child{node[Feuille] (12113) {}}
                                child{node[Feuille] (12114) {}}
                    }
                        child{node[Noeud] (12115) {}
                                child{node[Feuille] (12116) {}}
                                child{node[Feuille] (12117) {}}
                    }
                }
                child{node[Feuille] (12118) {}};
            \node[Cercle,fit=(1211)(12111)(12112)(12113)(12114)(12115)(12116)(12117)] (c1211) {};
            \draw[Arc] (c1111) -- (c1211);
            \node[Noeud] at (20, 12) (3211) {}
                child{node[Noeud] (32111) {}
                        child{node[Feuille] (32112) {}}
                        child{node[Noeud] (32113) {}
                                child{node[Feuille] (32114) {}}
                                child{node[Noeud] (32115) {}
                                        child{node[Feuille] (32116) {}}
                                        child{node[Feuille] (32117) {}}
                        }
                    }
                }
                child{node[Feuille] (32118) {}};
            \node[Cercle,fit=(3211)(32111)(32112)(32113)(32114)(32115)(32116)(32117)] (c3211) {};
            \draw[Arc] (c1211) -- (c3211);
            \node[Noeud] at (-10, 2) (2111) {}
                child{node[Noeud] (21111) {}
                        child{node[Noeud] (21112) {}
                                child{node[Feuille] (21113) {}}
                                child{node[Noeud] (21114) {}
                                        child{node[Feuille] (21115) {}}
                                        child{node[Feuille] (21116) {}}
                        }
                    }
                        child{node[Feuille] (21117) {}}
                }
                child{node[Feuille] (21118) {}};
            \node[Cercle,fit=(2111)(21111)(21112)(21113)(21114)(21115)(21116)(21117)] (c2111) {};
            \draw[Arc] (c1111) -- (c2111);
            \node[Noeud] at (-20, 4) (3111) {}
                child{node[Noeud] (31111) {}
                        child{node[Feuille] (31112) {}}
                        child{node[Noeud] (31113) {}
                                child{node[Noeud] (31114) {}
                                        child{node[Feuille] (31115) {}}
                                        child{node[Feuille] (31116) {}}
                        }
                                child{node[Feuille] (31117) {}}
                    }
                }
                child{node[Feuille] (31118) {}};
            \node[Cercle,fit=(3111)(31111)(31112)(31113)(31114)(31115)(31116)(31117)] (c3111) {};
            \draw[Arc] (c2111) -- (c3111);
            \draw[Arc] (c3111) -- (c3211);
            \node[Noeud] at (9.3, 14) (1311) {}
                child{node[Noeud] (13111) {}
                        child{node[Feuille] (13112) {}}
                        child{node[Feuille] (13113) {}}
                }
                child{node[Noeud] (13114) {}
                        child{node[Noeud] (13115) {}
                                child{node[Feuille] (13116) {}}
                                child{node[Feuille] (13117) {}}
                    }
                        child{node[Feuille] (13118) {}}
                };
            \node[Cercle,fit=(1311)(13111)(13112)(13113)(13114)(13115)(13116)(13117)] (c1311) {};
            \draw[Arc] (c1211) -- (c1311);
            \node[Noeud] at (1, 21) (1321) {}
                child{node[Noeud] (13211) {}
                        child{node[Feuille] (13212) {}}
                        child{node[Feuille] (13213) {}}
                }
                child{node[Noeud] (13214) {}
                        child{node[Feuille] (13215) {}}
                        child{node[Noeud] (13216) {}
                                child{node[Feuille] (13217) {}}
                                child{node[Feuille] (13218) {}}
                    }
                };
            \node[Cercle,fit=(1321)(13211)(13212)(13213)(13214)(13215)(13216)(13217)] (c1321) {};
            \draw[Arc] (c1311) -- (c1321);
            \draw[Arc] (c1121) -- (c1321);
            \node[Noeud] at (-10.75, 19) (2121) {}
                child{node[Noeud] (21211) {}
                        child{node[Feuille] (21212) {}}
                        child{node[Noeud] (21213) {}
                                child{node[Feuille] (21214) {}}
                                child{node[Feuille] (21215) {}}
                    }
                }
                child{node[Noeud] (21216) {}
                        child{node[Feuille] (21217) {}}
                        child{node[Feuille] (21218) {}}
                };
            \node[Cercle,fit=(2121)(21211)(21212)(21213)(21214)(21215)(21216)(21217)] (c2121) {};
            \draw[Arc] (c2111) -- (c2121);
            \draw[Arc] (c1121) -- (c2121);
            \node[Noeud] at (-20.9, 21) (4111) {}
                child{node[Feuille] (41111) {}}
                child{node[Noeud] (41112) {}
                        child{node[Noeud] (41113) {}
                                child{node[Noeud] (41114) {}
                                        child{node[Feuille] (41115) {}}
                                        child{node[Feuille] (41116) {}}
                        }
                                child{node[Feuille] (41117) {}}
                    }
                        child{node[Feuille] (41118) {}}
                };
            \node[Cercle,fit=(4111)(41111)(41112)(41113)(41114)(41115)(41116)(41117)] (c4111) {};
            \draw[Arc] (c3111) -- (c4111);
            \node[Noeud] at (19.1, 29) (4211) {}
            child{node[Feuille] (42111) {}}
            child{node[Noeud] (42112) {}
                    child{node[Noeud] (42113) {}
                            child{node[Feuille] (42114) {}}
                            child{node[Noeud] (42115) {}
                                    child{node[Feuille] (42116) {}}
                                    child{node[Feuille] (42117) {}}
                    }
                }
                    child{node[Feuille] (42118) {}}
            };
            \node[Cercle,fit=(4211)(42111)(42112)(42113)(42114)(42115)(42116)(42117)] (c4211) {};
            \draw[Arc] (c4111) -- (c4211);
            \draw[Arc] (c3211) -- (c4211);
            \node[Noeud] at (-11.35, 29) (4121) {}
                child{node[Feuille] (41211) {}}
                child{node[Noeud] (41212) {}
                        child{node[Noeud] (41213) {}
                                child{node[Feuille] (41214) {}}
                                child{node[Feuille] (41215) {}}
                    }
                        child{node[Noeud] (41216) {}
                                child{node[Feuille] (41217) {}}
                                child{node[Feuille] (41218) {}}
                    }
                };
            \node[Cercle,fit=(4121)(41211)(41212)(41213)(41214)(41215)(41216)(41217)] (c4121) {};
            \draw[Arc] (c4111) -- (c4121);
            \draw[Arc] (c2121) -- (c4121);
            \node[Noeud] at (8.7, 34) (4311) {}
                child{node[Feuille] (43111) {}}
                child{node[Noeud] (43112) {}
                        child{node[Feuille] (43113) {}}
                        child{node[Noeud] (43114) {}
                                child{node[Noeud] (43115) {}
                                        child{node[Feuille] (43116) {}}
                                        child{node[Feuille] (43117) {}}
                        }
                                child{node[Feuille] (43118) {}}
                    }
                };
            \node[Cercle,fit=(4311)(43111)(43112)(43113)(43114)(43115)(43116)(43117)] (c4311) {};
            \draw[Arc] (c1311) -- (c4311);
            \draw[Arc] (c4211) -- (c4311);
            \node[Noeud] at (0.55, 38) (4321) {}
                child{node[Feuille] (43211) {}}
                child{node[Noeud] (43212) {}
                        child{node[Feuille] (43213) {}}
                        child{node[Noeud] (43214) {}
                                child{node[Feuille] (43215) {}}
                                child{node[Noeud] (43216) {}
                                        child{node[Feuille] (43217) {}}
                                        child{node[Feuille] (43218) {}}
                        }
                    }
                };
            \node[Cercle,fit=(4321)(43211)(43212)(43213)(43214)(43215)(43216)(43217)] (c4321) {};
            \draw[Arc] (c4121) -- (c4321);
            \draw[Arc] (c1321) -- (c4321);
            \draw[Arc] (c4311) -- (c4321);
        \end{tikzpicture}
    }}}
    \caption{The Tamari lattices $\Tam_3$ and $\Tam_4$.}
    \label{FIG_Tamari_3_4}
\end{figure}
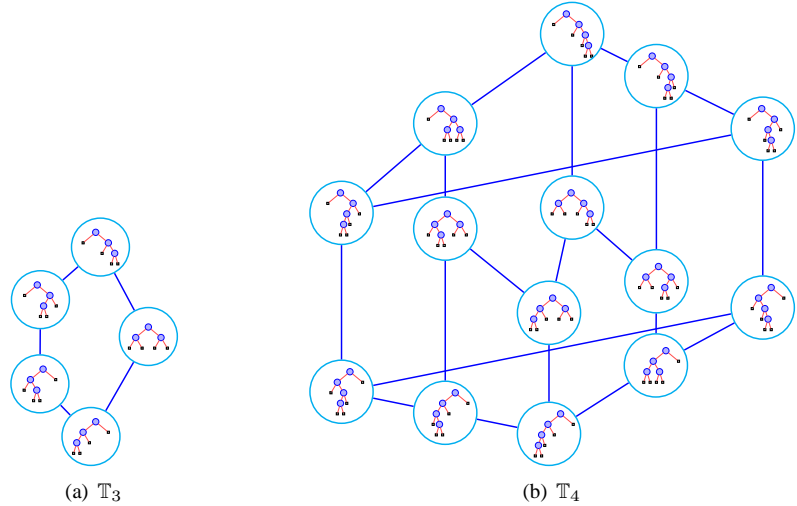

The Tamari relation is an order relation. For $n \geq 0$, the set $\EnsArb_n$
with the $\TamOrd$ order relation defines a lattice: the Tamari lattice. We
denote by $\Tam_n = (\EnsArb_n, \TamOrd)$ the Tamari lattice of order $n$.

\section{Closure by interval of the set of balanced trees} \label{Sec:Closure}

\subsection{Rotations and balance}

Let us first consider the modifications of the imbalance values of the nodes
of a tree $T_0 = (A \ArbCons B) \ArbCons C$ when a rotation at its root is
applied. Let $T_1$ be the tree obtained by this rotation, $y$ the root of
$T_0$ and $x$ the left child of $y$ in $T_0$. Note first that the imbalance
values of the nodes of the trees $A$, $B$ and $C$ are not modified by the
rotation. Indeed, only the imbalance values of the nodes $x$ and $y$ are
changed. Since $T_0$ is balanced, we have $\Des_{T_0}(x) \in \{-1, 0, 1\}$
and $\Des_{T_0}(y) \in \{-1, 0, 1\}$. Thus, the pair $(\Des_{T_0}(x), \Des_{T_0}(y))$
can take nine different values. Here follows the list of the imbalance
values of the nodes $x$ and $y$ in the trees $T_0$ and $T_1$:
\begin{table}[ht]
    \centering
    \begin{tabular}{c|ccccccccc}
                                         & (B1)              & (U1)    & (U2)    & (B2)             & (U3)   & (U4)   & (U5)    & (U6)   & (U7)   \\ \hline
        $\left(\Des_{T_0}(x), \Des_{T_0}(y)\right)$ & \textbf{(-1, -1)} & (-1, 0) & (-1, 1) & \textbf{(0, -1)} & (0, 0) & (0, 1) & (1, -1) & (1, 0) & (1, 1) \\
        $\left(\Des_{T_1}(x), \Des_{T_1}(y)\right)$ & \textbf{(1, 1)}   & (2, 2)  & (3, 3)  & \textbf{(1, 0)}  & (2, 1) & (3, 2) & (2, 0)  & (3, 1) & (4, 2)
    \end{tabular}
    \caption{Imbalance values of the nodes $x$ and $y$ in $T_0$ and $T_1$.}
\end{table}

Notice that only in (B1) and (B2) the tree $T_1$ is balanced. We have the
following lemma:

\begin{lemma} \label{LEMhauteurEgaleEq}
    Let $T_0$ and $T_1$ be two balanced trees such that $T_0 \TamPar T_1$.
    Then, the trees $T_0$ and $T_1$ have the same height.
\end{lemma}
\begin{proof}
    Since $T_0$ and $T_1$ are both balanced, the rotation modifies a subtree
    $S_0$ of $T_0$ such that the imbalance values of the root of $S_0$,
    namely $y$, and the left child of $y$, namely $x$, satisfy (B1) or (B2).
    Let $S_1$ be the tree obtained by the rotation of root $y$ from $S_0$.
    Computing the height of the trees $S_0$ and $S_1$, we have $\Ht(S_0) = \Ht(S_1)$.
    Thus, as a rotation modifies a tree locally, we have $\Ht(T_0) = \Ht(T_1)$.
\end{proof}

A rotation transforming a tree $T_0$ into a tree $T_1$ is a \emph{conservative
balancing rotation} if both $T_0$ and $T_1$ are balanced. Considering $y$
the root of this rotation and $x$ the left child of $y$, we see, by the
previous computations and Lemma \ref{LEMhauteurEgaleEq}, that $T_0$ and
$T_1$ are both balanced if and only if $T_0$ is balanced and
\begin{equation}
    (\Des_{T_0}(x), \Des_{T_0}(y)) \in \{(-1, -1), (0, -1)\}.
\end{equation}
Similarly, a rotation is an \emph{unbalancing rotation} if $T_0$ is balanced
but $T_1$ not.

\begin{lemma} \label{LEMrotationDes}
    Let $T_0$ be a balanced tree and $T_1$ be an unbalanced tree such that
    $T_0 \TamPar T_1$. Then, there exists a node $z$ in $T_1$ such that
    $\Des_{T_1}(z) \geq 2$ and the left subtree and the right subtree of
    $z$ are both balanced.
\end{lemma}
\begin{proof}
    Immediate, looking at (U1), (U2), (U3), (U4), (U5), (U6) and (U7).
\end{proof}

\subsection{Admissible words}

\begin{definition}
    A word $z \in \EnsNat^*$ is \emph{admissible} if either $|z| \leq 1$
    or we have $z_1 - 1 \leq z_2$, and the word obtained by applying the
    substitution
    \begin{equation} \label{SUBSS}
        z_1 . z_2 \longrightarrow
        \begin{cases}
            \max \{z_1, z_2\} + 1 & \text{if $z_1 - 1 \leq z_2 \leq z_1 + 1$,} \\
            z_2 & \text{otherwise}
        \end{cases}
    \end{equation}
    to $z$ is admissible. Let us denote by $\EnsAdmissible$ the set of
    admissible words.
\end{definition}

For example, we can check that the word $z = 00122$ is admissible. Indeed,
applying the substitution (\ref{SUBSS}), we have $00122 \rightarrow 1122
\rightarrow 222 \rightarrow 32 \rightarrow 4$ and at each step, the condition
$z_1 - 1 \leq z_2$ holds. The word $z' = 1234488$ is also admissible: $1234488
\rightarrow 334488 \rightarrow 44488 \rightarrow 5488 \rightarrow 688
\rightarrow 88 \rightarrow 9$. The word $z'' = 3444$ is not admissible
because we have $3444 \rightarrow 544 \rightarrow 64$ and since that
$6 - 1 \nleq 4$, we have $z'' \notin \EnsAdmissible$.

\begin{remark} \label{REM_croiss}
    If $z$ is an admissible word, then, for all $1 \leq i \leq |z| - 1$
    the inequality $z_i - 1 \leq z_{i + 1}$ holds.
\end{remark}

\begin{remark} \label{REM_prefsuff}
    The prefixes and suffixes of an admissible word are still admissible.
\end{remark}

\begin{remark} \label{REM_sub}
    If $z = u.v$ where $z, u, v \in \EnsNat^*$ are admissible words, after
    applying the substitution (\ref{SUBSS}) to $v$ to obtain the word $v'$,
    the word $z' = u.v'$ is still admissible.
\end{remark}

Let the \emph{potential} $\Potentiel(z)$ of an admissible word $z$ be the
outcome of the application of the substitution (\ref{SUBSS}). In the
previous examples, we have $\Potentiel(z) = 4$ and $\Potentiel(z') = 9$.

Let $T$ be a tree, $x$ be a node of $T$, $(x = x_1, x_2, \ldots, x_\ell)$ be
the sequence of all ancestors of $x$ whose right sons are not themselves
ancestors of $x$, ordered from bottom to top and $(S_{x_i})_{1 \leq i \leq \ell}$
be the sequence of the right subtrees of the nodes $x_i$ (see Figure
\ref{FIGArbreDroite}). The word $z$ on the alphabet $\EnsNat$ defined by
$z_i = \Ht(S_{x_i})$ is called the \emph{characteristic word} of the node
$x$ in the tree $T$ and denoted by $\MotCar_T(x)$.

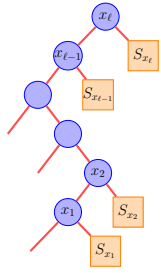
\begin{figure}[ht]
    \centering
    \scalebox{.40}{
        \begin{tikzpicture}
            \tikzstyle{level} = [level distance = 13mm]
            \tikzstyle{level 1} = [sibling distance=25mm]
            \tikzstyle{level 2} = [sibling distance=20mm]
            \tikzstyle{level 3} = [sibling distance=20mm]
            \tikzstyle{level 4} = [sibling distance=20mm]
            \tikzstyle{level 5} = [sibling distance=20mm]
            \tikzstyle{level 6} = [sibling distance=25mm]
            \tikzstyle{Noeud} = [circle, draw = blue!100, fill = blue!30, thick, inner sep = 0pt, minimum size = 9mm]
            \tikzstyle{SArbre} = [rectangle, draw = orange!100, fill = orange!30, thick, inner sep = 0pt, minimum size = 10mm]
            \node[Noeud]{\Large $x_\ell$}
                child{node[Noeud]{\Large $x_{\ell - 1}$}
                    child{node[Noeud]{}
                            child[Noeud]
                            child{node[Noeud]{}
                                child[Noeud]
                                child{node[Noeud]{\Large $x_2$}
                                    child{node[Noeud]{\Large $x_1$}
                                        child[Noeud]
                                        child{node[SArbre]{\Large $S_{x_1}$}}
                                    }
                                    child{node[SArbre]{\Large $S_{x_2}$}}
                                }
                            }
                    }
                    child{node[SArbre]{\Large $S_{x_{\ell - 1}}$}}
                }
                child{node[SArbre]{\Large $S_{x_\ell}$}};
        \end{tikzpicture}
    }
    \caption{The sequence $(S_{x_i})_{1 \leq i \leq \ell}$ associated to the node $x = x_1$.}
    \label{FIGArbreDroite}
\end{figure}

\begin{lemma} \label{LEM11}
    Let $T$ be a balanced tree, $x$ a node of $T$, and $z$ the characteristic
    word of $x$. Then, $z$ is admissible and $\Potentiel(z) \leq \Ht(T)$.
\end{lemma}
\begin{proof}
    By structural induction on balanced trees. The lemma is obviously true
    for the trees of the set $\EnsEq_0 \cup \EnsEq_1$. Let $L$ and $R$ be
    two balanced trees such that $T = L \ArbCons R$ is balanced too and
    assume that the lemma is true for both $L$ and $R$. Let $x$ be a node
    of $T$. Distinguishing the cases where $x$ is a node of $L$, a node
    of $R$, or the root of $T$, we have, by induction, the statement of
    the lemma.
\end{proof}

\begin{lemma} \label{LEM2}
    Let $T$ be a tree and $y$ a node of $T$ such that $\MotCar_T(y)$ is
    admissible and all subtrees of the sequence $(S_{y_i})_{1 \leq i \leq \ell}$
    are balanced. Then, for all node $x$ of $T$ such that $y \ADroite_T x$,
    the word $\MotCar_T(x)$ is admissible.
\end{lemma}
\begin{proof}
    If $x$ is an ancestor of $y$, the word $\MotCar_T(x)$ is a suffix of
    $\MotCar_T(y)$, thus we have, by Remark \ref{REM_prefsuff}, $\MotCar_T(x) \in \EnsAdmissible$.
    Otherwise, let $S$ be the subtree of $T$ such that $x$ is a node of
    $S$ and the parent of $S$ in $T$ is an ancestor of $y$. We have
    $\MotCar_T(y) = u.\Ht(S).v$ where $u, v \in \EnsAdmissible$. As
    $y \ADroite_T S$, we have $S \in \EnsEq$ and by Lemma \ref{LEM11}, we
    have $\MotCar_S(x) \in \EnsAdmissible$ and $\Potentiel(\MotCar_S(x)) \leq \Ht(S)$.
    Thus, thanks to Remark \ref{REM_prefsuff}, $\Ht(S).v \in \EnsAdmissible$,
    so that $\MotCar_T(x) = \MotCar_S(x).v \in \EnsAdmissible$.
\end{proof}

\subsection{The main result}

\begin{theorem} \label{THE1}
    Let $T_0$ and $T_1$ be two balanced trees such that $T_0 \TamOrd T_1$.
    Then, the interval $[T_0, T_1]$ only contains balanced trees. In other
    words, all successors of a tree obtained doing an unbalancing rotation
    into a balanced tree are unbalanced.
\end{theorem}
\begin{proof}
    To prove the theorem, we shall show that for all balanced tree $T_0$ and an
    unbalanced tree $T_1$ such that $T_0 \TamPar T_1$, all trees $T_2$
    such that $T_1 \TamOrd T_2$ are unbalanced. Indeed, $T_1$ has a property
    guaranteeing it is unbalanced that can be kept for all its successors.

    Let $\PropDes_T(x)$ be the property: the node $x$ of $T$ and the node $y$ which is
    the leftmost node of the left subtree of $x$ satisfy: (see Figure \ref{FIGPropDes}):
    \begin{enumerate}[(1)]
        \item $\Des_T(x) \geq 2$;
        \item the left subtree of $x$ is balanced;
        \item all the subtrees $S$ such that $y \ADroite_T S$ are balanced;
        \item $\MotCar_T(y) \in \EnsAdmissible$.
    \end{enumerate}
    Point (2) guarantees that each tree having the previous property
    is unbalanced.
    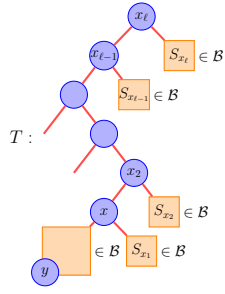
\begin{figure}[ht]
        \centering
            \scalebox{.40}{
                \begin{tikzpicture}
                    \tikzstyle{level} = [level distance = 13mm]
                    \tikzstyle{level 1} = [sibling distance=25mm]
                    \tikzstyle{level 2} = [sibling distance=20mm]
                    \tikzstyle{level 3} = [sibling distance=20mm]
                    \tikzstyle{level 4} = [sibling distance=20mm]
                    \tikzstyle{level 5} = [sibling distance=20mm]
                    \tikzstyle{level 6} = [sibling distance=25mm]
                    \tikzstyle{Noeud} = [circle, draw = blue!100, fill = blue!30, thick, inner sep = 0pt, minimum size = 9mm]
                    \tikzstyle{SArbre} = [rectangle, draw = orange!100, fill = orange!30, thick, inner sep = 0pt, minimum size = 10mm]
                    \tikzstyle{SArbreY} = [rectangle, draw = orange!100, fill = orange!30, thick, inner sep = 0pt, minimum size = 16mm]
                    \node[Noeud]{\Large $x_\ell$}
                        child{node[Noeud]{\Large $x_{\ell - 1}$}
                            child{node[Noeud]{}
                                    child[Noeud]
                                    child{node[Noeud]{}
                                        child[Noeud]
                                        child{node[Noeud]{\Large $x_2$}
                                            child{node[Noeud]{\Large $x$}
                                                child{node[SArbreY, label=right:\Large $\in \EnsEq$] (ssy) {}
                                                    node[Noeud] [below left of = ssy] {\Large $y$}
                                                }
                                                child{node[SArbre, label=right:\Large $\in \EnsEq$]{\Large $S_{x_1}$}}
                                            }
                                            child{node[SArbre, label=right:\Large $\in \EnsEq$]{\Large $S_{x_2}$}}
                                        }
                                    }
                            }
                            child{node[SArbre, label=right:\Large $\in \EnsEq$]{\Large $S_{x_{\ell - 1}}$}}
                        }
                        child{node[SArbre, label=right:\Large $\in \EnsEq$]{\Large $S_{x_\ell}$}};
                    \node at (-4, -4) {\LARGE $T:$};
                \end{tikzpicture}
            }
        \caption{The imbalance property $\PropDes_T(x)$.
                 The node $y$ is the leftmost node of the left subtree of the node $x$.}
        \label{FIGPropDes}
    \end{figure}

    First, let us show that there exists a node $x$ such that $\PropDes_{T_1}(x)$
    is true. The tree $T_1$ is obtained by an unbalancing rotation from
    $T_0$. By Lemma \ref{LEMrotationDes}, there exists a node $x$ in $T_1$
    satisfying points (1) and (2). As the left and right subtrees of $x$
    are balanced and as all the trees on the right compared to $x$ are
    balanced in $T_0$, they remain balanced in $T_1$, so that point (3)
    checks out. To establish (4), denoting by $y$ the leftmost node of the
    left subtree of $x$ in $T_1$, we have, by Remark \ref{REM_sub} and
    Lemmas \ref{LEM11} and \ref{LEM2}, $\MotCar_{T_1}(y) \in \EnsAdmissible$.

    Now, let us show that given a tree $T_1$ such that $\PropDes_{T_1}(x)$
    is satisfied for a node $x$ of $T_1$, for all tree $T_2$ such that
    $T_1 \TamPar T_2$, there exists a node $x'$ of $T_2$ such  that
    $\PropDes_{T_2}(x')$ is satisfied. Let $y$ be the leftmost node of the
    left subtree of $x$ in $T_1$ and $r$ be the root of the rotation that
    transforms $T_1$ into $T_2$. We will treat all cases depending on the
    position of $r$ compared to $y$.

    If the node $r$ belongs to a subtree of $T_1$ which is on the left
    compared to $y$, the rotation does not modify any of the subtrees on
    the right compared to $y$. Thus we have $\PropDes_{T_2}(x)$.

    If the subtree $S_1$ of root $r$ satisfies $y \ADroite_{T_1} S_1$, let $S_2$
    be the subtree of $T_2$ obtained by the rotation of $S_1$ which transforms
    $T_1$ into $T_2$. If $S_2$ is balanced, by Lemma \ref{LEMhauteurEgaleEq},
    $\Ht(S_1) = \Ht(S_2)$ and we have $\PropDes_{T_2}(x)$. If $S_2$ is not
    balanced, by the study of the initial case, we have $\PropDes_{S_2}(x')$
    for a node $x'$ of $S_2$. Besides, by Remark \ref{REM_sub} and Lemma
    \ref{LEM2}, denoting by $y'$ the leftmost node of the left subtree
    of $x'$ in $T_2$, we have $\MotCar_{T_2}(y') \in \EnsAdmissible$ and thus,
    $\PropDes_{T_2}(x')$.

    If the node $r$ is an ancestor of $y$ and the left child of $r$ is
    still an ancestor of $y$, let $B$ be the right subtree of $r$ and $A$
    the right subtree of the left child of $r$ in $T_1$. The rotation
    replaces the trees $A$ and $B$ by the tree $A \ArbCons B$. As
    $\MotCar_{T_1}(y) \in \EnsAdmissible$, we have, by Remark \ref{REM_croiss},
    $\Ht(A) - 1 \leq \Ht(B)$. Thus, if $A \ArbCons B$ is balanced, we have
    $\PropDes_{T_2}(x)$. Indeed, points (1), (2) and (3) are clearly satisfied
    and, by Remark \ref{REM_sub}, we have (4). If $A \ArbCons B$ is unbalanced,
    calling $x'$ the root of this tree in $T_2$, we have $\Des_{T_2}(x') \geq 2$,
    and, calling $y'$ the leftmost node of $A$, we have, by Lemma \ref{LEM2},
    $\MotCar_{T_2}(y') \in \EnsAdmissible$. Thus we have $\PropDes_{T_2}(x')$.

    If the node $r$ is an ancestor of $y$ and the right child of $r$ is
    still an ancestor of $y$, the rotation does not modify any of the
    subtrees on the right compared to $y$. Thus, we have $\PropDes_{T_2}(x)$.
\end{proof}

\section{Tree patterns and synchronous grammars} \label{Sec:Grammars}

Word patterns are usually used to describe languages by considering the
set of words avoiding them. We use the same idea to describe sets of trees.
We show first that we can describe two interesting subsets of the set of
balanced trees only by two-nodes patterns.

Next, we follow the methods of ~\cite{KNU398, BLL94} to characterize, in
our setting, a way to obtain a functional equation admitting as fixed
point the generating series enumerating balanced trees. In this purpose,
we introduce \emph{synchronous grammars}, allowing to generate trees iteratively.
This method gives us a way to enumerate trees avoiding a set of tree patterns
because, as we shall see, functional equations of generating series can
be extracted from synchronous grammars.

\subsection{Tree patterns}

\begin{definition}
    A \emph{tree pattern} is a nonempty non complete rooted planar binary
    tree with labels in $\EnsRel$.
\end{definition}

Let $T$ be a tree and $T_{\Des}$ be the labeled tree of shape $T$ where
each node of $T_{\Des}$ is labeled by its imbalance value. The tree $T$
admits an \emph{occurrence} of a tree pattern $p$ if a connected component
of $T_{\Des}$  has the same shape and same labels as $p$.

Now, given a set $P$ of tree patterns, we can define the set composed of
the trees that do not admit any occurrence of the elements of $P$. For
example, the set
\begin{equation}
    \left \{ \Noeud{.45}{$i$} ~|~ i \notin \{-1, 0, 1\} \right \}
\end{equation}
describes the set of balanced trees; the set
\begin{equation}
    \left \{ \Noeud{.45}{$i$} ~|~ i \ne 0 \right \}
\end{equation}
describes the set of perfect trees and
\begin{equation}
    \left \{ \raisebox{-0.75em}{\MotifA{.45}{$i$}{$j$}} ~|~ i, j \in \EnsRel \right \}
\end{equation}
describes the set of right comb trees.

\subsection{Two particular subsets of balanced trees}

Let us describe a subset of the balanced trees and its counterpart such
that its elements are, roughly speaking, at the end of the balanced trees
subset in the Tamari lattice:

\begin{definition}
    A balanced tree $T_0$ (resp. $T_1$) is \emph{maximal} (resp.
    \emph{minimal}) if for all balanced tree $T_1$ (resp. $T_0$) such
    that $T_0 \TamPar T_1$ we have $T_1$ (resp. $T_0$) unbalanced.
\end{definition}

\begin{proposition} \label{PRO1}
    A balanced tree $T$ is maximal if and only if it avoids the set of
    tree patterns
    \begin{equation}
        P_{\textnormal{max}} := \left \{
            \raisebox{-0.75em}{\MotifA{.45}{$-1$}{$-1$}},~ \raisebox{-0.75em}{\MotifA{.45}{$0$}{$-1$}}
        \right \}.
    \end{equation}
    Similarly, a balanced tree $T$ is minimal if and only if it avoids the set of
    tree patterns
    \begin{equation}
        P_{\textnormal{min}} := \left \{
            \raisebox{-0.75em}{\MotifB{.45}{$1$}{$1$}},~ \raisebox{-0.75em}{\MotifB{.45}{$1$}{$0$}}
        \right \}.
    \end{equation}
\end{proposition}
\begin{proof}
    Assume that $T$ is maximal. For all tree $T_1$ such that $T \TamPar T_1$
    we have $T_1$ unbalanced. Thus, it is impossible to do a conservative
    balancing rotation from $T$ and it avoids the set $P_{\textnormal{max}}$.

    Assume that $T$ avoids the two tree patterns of $P_{\textnormal{max}}$,
    then, for every tree $T_1$ such that $T \TamPar T_1$, the tree $T_1$
    is unbalanced because we can do only unbalancing rotations in $T$.
    Thus, the tree $T$ is maximal.

    The proof of the second part of the proposition is done in an analogous way.
\end{proof}

\subsection{Synchronous grammars and enumeration of balanced trees}

Let us first describe a way to obtain the functional equation admitting
as fixed point the generating series which enumerates balanced trees~\cite{KNU398, BLL94}.

The idea is to generate trees by allowing them to grow from the root to
the leaves step by step. For that, we generate \emph{bud trees}, that are
non complete rooted planar binary trees with the particularity that the
set of external nodes (the nodes without descendant) are \emph{buds}. A
bud tree grows by \emph{simultaneously} substituting all of its buds by
new bud trees. Trees are finally obtained replacing buds by leaves. The
rules of substitution allowing to generate bud trees form a \emph{synchronous
grammar}. The link between tree patterns and synchronous grammars is that
synchronous grammars generate trees controlling the imbalance value of
the nodes. The rules generating balanced trees are
\begin{eqnarray}
    \raisebox{0.65em}{\Bourgeon{.45}{$x$}} & \raisebox{0.75em}{$\longrightarrow$}&
            \BourgeonA{.45}{$x$}{-1}{$y$} ~\raisebox{0.75em}{$+$}~
            \BourgeonA{.45}{$x$}{$0$}{$x$} ~\raisebox{0.75em}{$+$}~
            \BourgeonA{.45}{$y$}{$1$}{$x$} \\
    \Bourgeon{.45}{$y$} & \raisebox{0.10em}{$\longrightarrow$} & \Bourgeon{.45}{$x$}
\end{eqnarray}

The role of the bud \Bourgeon{.45}{$x$} is to generate a node which has $-1$,
$0$ or $1$ as imbalance value, the only values that a balanced tree can
have. The role of the bud \Bourgeon{.45}{$y$} is to delay the growth of
the bud tree to enable the creation of the imbalance values $-1$ and $1$.
We have the following theorem:

\begin{theorem}
    Let $B$ be a bud tree generated from the bud \Bourgeon{.45}{$x$}
    by the previous synchronous grammar. If $B$ does not contain any
    bud \Bourgeon{.45}{$y$}, replacing all buds \Bourgeon{.45}{$x$}
    by leaves, we obtain a tree $T$ where each node $z$ of $T$ is labeled
    by $\Des_T(z)$. In this way, the previous synchronous grammar generates
    exactly the set of balanced trees.
\end{theorem}

Figure \ref{fig:Generation} shows an example of such a generation.
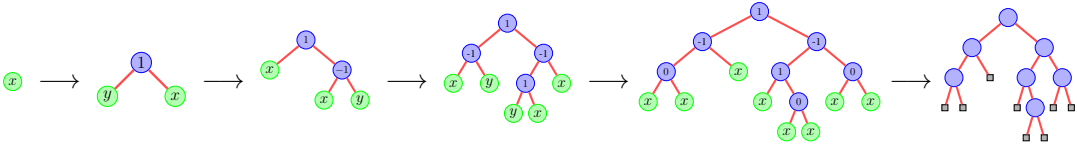
\begin{figure}[ht]
    \centering
        \raisebox{1.9em}{
            \scalebox{.4}{
            \begin{tikzpicture}
                \node[Bourgeon]{\Large $x$};
            \end{tikzpicture}
            }
        }
        \raisebox{2em}{$\longrightarrow$}
        \raisebox{1.3em}{
            \scalebox{.45}{
            \begin{tikzpicture}
                \node[Bourgeon](1)at(0,0){\Large $y$};
                \node[Noeud](2)at(1,1){\Large $1$};
                \node[Bourgeon](3)at(2,0){\Large $x$};
                \draw[Arrete](1)--(2);
                \draw[Arrete](2)--(3);
            \end{tikzpicture}
            }
        }
        \raisebox{2em}{$\longrightarrow$}
        \raisebox{1.2em}{
            \scalebox{.4}{
            \begin{tikzpicture}
                \node[Noeud]{$1$}
                    child{node[Bourgeon]{\Large $x$}
                            child[missing]
                            child[missing]
                    }
                    child{node[Noeud]{$-1$}
                            child{node[Bourgeon]{\Large $x$}
                                    child[missing]
                                    child[missing]
                        }
                            child{node[Bourgeon]{\Large $y$}
                                    child[missing]
                                    child[missing]
                        }
                    };
            \end{tikzpicture}
            }
        }
        \raisebox{2em}{$\longrightarrow$}
        \raisebox{0.7em}{
            \scalebox{.4}{
            \begin{tikzpicture}
                \tikzstyle{level 1} = [sibling distance=24mm]
                \tikzstyle{level 2} = [sibling distance=12mm]
                \tikzstyle{level 3} = [sibling distance=8mm]
                \node[Noeud]{$1$}
                    child{node[Noeud]{-1}
                            child{node[Bourgeon]{\Large $x$}
                                    child[missing]
                                    child[missing]
                        }
                            child{node[Bourgeon]{\Large $y$}
                                    child[missing]
                                    child[missing]
                        }
                    }
                    child{node[Noeud]{-1}
                            child{node[Noeud]{$1$}
                                    child{node[Bourgeon]{\Large $y$}
                                            child[missing]
                                            child[missing]
                            }
                                    child{node[Bourgeon]{\Large $x$}
                                            child[missing]
                                            child[missing]
                            }
                        }
                            child{node[Bourgeon]{\Large $x$}
                                    child[missing]
                                    child[missing]
                        }
                    };
            \end{tikzpicture}
            }
        }
        \raisebox{2em}{$\longrightarrow$}
        \scalebox{.4}{
        \begin{tikzpicture}
            \tikzstyle{level 1} = [sibling distance=38mm]
            \tikzstyle{level 2} = [sibling distance=24mm]
            \tikzstyle{level 3} = [sibling distance=12mm]
            \tikzstyle{level 4} = [sibling distance=8mm]
            \node[Noeud]{$1$}
                child{node[Noeud]{-1}
                        child{node[Noeud]{$0$}
                                child{node[Bourgeon]{\Large $x$}
                                        child[missing]
                                        child[missing]
                        }
                                child{node[Bourgeon]{\Large $x$}
                                        child[missing]
                                        child[missing]
                        }
                    }
                        child{node[Bourgeon]{\Large $x$}
                                child[missing]
                                child[missing]
                    }
                }
                child{node[Noeud]{-1}
                        child{node[Noeud]{$1$}
                                child{node[Bourgeon]{\Large $x$}
                                        child[missing]
                                        child[missing]
                        }
                                child{node[Noeud]{$0$}
                                        child{node[Bourgeon]{\Large $x$}
                                                child[missing]
                                                child[missing]
                            }
                                        child{node[Bourgeon]{\Large $x$}
                                                child[missing]
                                                child[missing]
                            }
                        }
                    }
                        child{node[Noeud]{$0$}
                                child{node[Bourgeon]{\Large $x$}
                                        child[missing]
                                        child[missing]
                        }
                                child{node[Bourgeon]{\Large $x$}
                                        child[missing]
                                        child[missing]
                        }
                    }
                };
        \end{tikzpicture}
        }
        \raisebox{2em}{$\longrightarrow$}
        \scalebox{.4}{
        \begin{tikzpicture}
            \node[Noeud]{}
                child{node[Noeud]{}
                        child{node[Noeud]{}
                                child{node[Feuille]{}}
                                child{node[Feuille]{}}
                    }
                        child{node[Feuille]{}}
                }
                child{node[Noeud]{}
                        child{node[Noeud]{}
                                child{node[Feuille]{}}
                                child{node[Noeud]{}
                                        child{node[Feuille]{}}
                                        child{node[Feuille]{}}
                        }
                    }
                        child{node[Noeud]{}
                                child{node[Feuille]{}}
                                child{node[Feuille]{}}
                    }
                };
        \end{tikzpicture}
        }
    \caption{Generation of a balanced tree.}
    \label{fig:Generation}
\end{figure}

The main purpose of synchronous grammars is to obtain a way to enumerate
the trees generated. We can translate the set of rules to obtain a functional
equation of the generating series enumerating them. For balanced trees, we
have~\cite{KNU398, BLL94, SLO08}:

\begin{theorem} \label{THEequationArbresEq}
    The generating series enumerating balanced trees according
    to the number of leaves of trees is $G_{\textnormal{bal}}(x) := A(x, 0)$ where
    \begin{equation}
        A(x, y) := x + A(x^2 + 2xy, x).
    \end{equation}
\end{theorem}

The resolution, or, in other words, the coefficient extraction for this
kind of functional equation, is made by iteration. We proceed by computing
the sequence of polynomials $(A_i)_{i \geq 0}$ defined by:
\begin{equation}
    A_i(x, y) =
    \begin{cases}
        x                         & \mbox{if $i = 0$,} \\
        x + A_{i-1}(x^2 + 2xy, x) & \mbox{otherwise.}
    \end{cases}
\end{equation}
The first iterations give
\begin{eqnarray}
    A_0 & = & x, \\
    A_1 & = & x + 2xy + x^2, \\
    A_2 & = & x + 2xy + x^2 + 4x^2y + 2x^3 + 4x^2y^2 + 4x^3y + x^4.
\end{eqnarray}

The fixed point of the sequence $(A_i)_{i \geq 0}$, after substituting
$0$ to the parameter $y$ in order to ignore bud trees with some buds
\Bourgeon{.45}{$y$}, is the generating series of balanced
trees counted according to the number of leaves.

We can refine this idea to enumerate maximal balanced trees:

\begin{proposition} \label{PROserieGenArbresEqMax}
    The generating series enumerating maximal balanced trees
    according to the number of leaves of the trees is $G_{\textnormal{max}}(x) := A(x, 0, 0)$
    where
    \begin{equation}
        A(x, y, z) := x + A(x^2 + xy + yz, x, xy).
    \end{equation}
\end{proposition}
\begin{proof}
    To obtain this functional equation, let us use the following synchronous
    grammar which generates maximal balanced trees:
    \begin{eqnarray}
        \raisebox{0.65em}{\Bourgeon{.45}{$x$}} & \raisebox{0.75em}{$\longrightarrow$} &
                                           \BourgeonA{.45}{$x$}{$0$}{$x$} ~\raisebox{0.75em}{$+$}~
                                           \BourgeonA{.45}{$y$}{$1$}{$x$} ~\raisebox{0.75em}{$+$}~
                                           \BourgeonA{.45}{$z$}{-1}{$y$} \\
        \Bourgeon{.45}{$y$} & \raisebox{0.10em}{$\longrightarrow$} & \Bourgeon{.45}{$x$} \\
        \raisebox{0.65em}{\Bourgeon{.45}{$z$}} & \raisebox{0.75em}{$\longrightarrow$} & \BourgeonA{.45}{$y$}{$1$}{$x$}
    \end{eqnarray}

    This grammar must generate only maximal balanced trees. By Proposition
    \ref{PRO1}, the generated trees must avoid the two tree patterns of
    $P_{\textnormal{max}}$. To do that, we have to control the growth of
    the bud \Bourgeon{.45}{$x$} when it generates a tree $S$
    such that its root has an imbalance value of $-1$. Indeed, if the root
    of the left subtree of $S$ grows with an imbalance value of $-1$ or $0$,
    one of the two tree patterns is not avoided. The idea is to force the
    imbalance value of the root of left subtree of $S$ to be $1$, role
    played by the bud \Bourgeon{.45}{$z$}.
\end{proof}

The solution of this functional equation give us the following first values
for the number of maximal trees in the Tamari lattice:
$1$, $1$, $1$, $1$, $2$, $2$, $2$, $4$, $6$, $9$, $11$, $13$, $22$,
$38$, $60$, $89$, $128$, $183$, $256$, $353$, $512$, $805$, $1336$, $2221$,
$3594$, $5665$, $8774$, $13433$, $20359$.

\section{The shape of the balanced tree intervals} \label{Sec:Shape}

\subsection{Isomorphism between balanced tree intervals and hypercubes}

A hypercube of dimension $k$ can be seen as a poset whose elements are
subsets of a set $\{e_1, \ldots, e_k \}$ ordered by the relation of inclusion.
Let us denote by $\HyperCube_k$ the hypercube poset of dimension $k$.

We have the following characterization of the shape of balanced tree intervals:

\begin{theorem} \label{THEintervalleHypercube}
    Let $T_0$ and $T_1$ be two balanced trees such that $T_0 \TamOrd T_1$.
    Then there exists $k \geq 0$ such that the posets $([T_0, T_1], \TamOrd)$
    and $\HyperCube_k$ are isomorphic.
\end{theorem}
\begin{proof}
    First, note by Theorem \ref{THE1}, that $I = [T_0, T_1] \subseteq \EnsEq$.
    Thus, every covering relation of the interval $I$ is a conservative
    balancing rotation.

    Then, note that the rotations needed to transform $T_0$ into $T_1$
    are disjoint in the sense that if $y$ is a node of $T_2 \in I$ and
    $x$ its left child, if we apply a conservative balancing rotation
    of root $y$ in $T_2$ to obtain $T_3 \in I$, all the rotations in the
    successors of $T_3$ of root $y$ and of root $x$ are unbalancing
    rotations. Indeed, by Lemma \ref{LEMhauteurEgaleEq}, each conservative
    balancing rotation modifies only the imbalance values of the root of
    the rotation and its left child, and, according to the values obtained,
    these two nodes cannot thereafter be roots of conservative balancing
    rotations.

    Besides, by the nature of the conservative balancing rotations and by
    Theorem \ref{THE1}, we can see that all the ways to transform $T_0$
    into $T_1$ solicit the same rotations, possibly in a different order.

    Now, we can associate to a tree $T \in I$ a subset of $\EnsNat$ containing
    the positions in the infix order of the nodes $y$ such that, to obtain
    $T$ from $T_0$, we have done, among other, a rotation of root $y$.
    The interval $I$ is isomorphic to the poset $\HyperCube_k$ where $k$
    is the number of rotations needed to transform $T_0$ into $T_1$.
\end{proof}

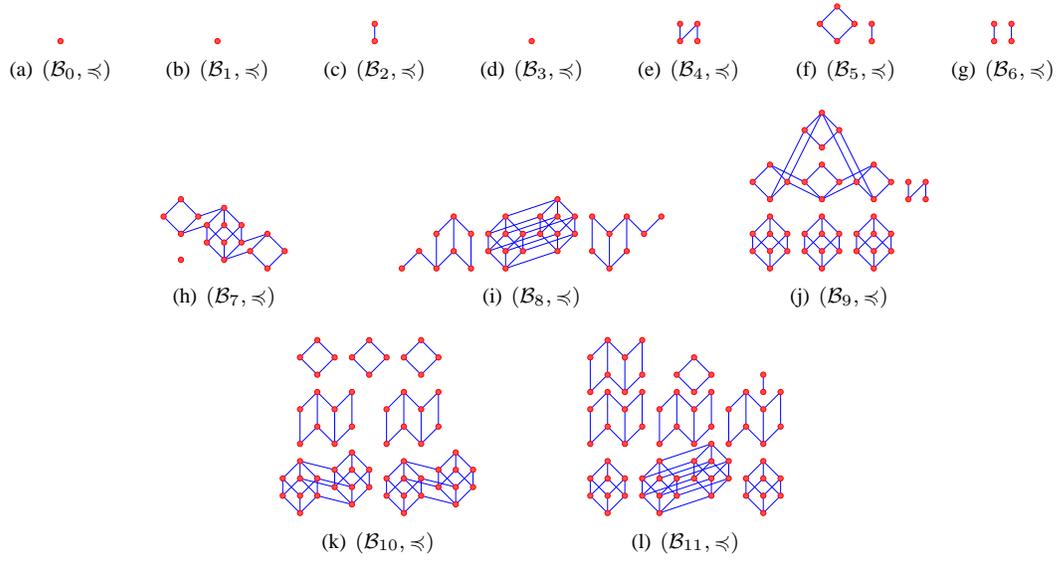
\begin{figure}[ht]
    \centering
    \subfigure[$(\EnsEq_0, \TamOrd)$]{\makebox[2cm]{\scalebox{.23}{
        \begin{tikzpicture}
            \node[Cercle] at (0, 0)  (000) {};
        \end{tikzpicture}
    }}}
    \subfigure[$(\EnsEq_1, \TamOrd)$]{\makebox[2cm]{\scalebox{.23}{
        \begin{tikzpicture}
            \node[Cercle] at (0, 0)  (000) {};
        \end{tikzpicture}
    }}}
    \subfigure[$(\EnsEq_2, \TamOrd)$]{\makebox[2cm]{\scalebox{.23}{
        \begin{tikzpicture}
            \node[Cercle] at (0, 0)  (000) {};
            \node[Cercle] at (0, 1)  (010) {};
            \draw[ArreteCube] (000) -- (010);
        \end{tikzpicture}
    }}}
    \subfigure[$(\EnsEq_3, \TamOrd)$]{\makebox[2cm]{\scalebox{.23}{
        \begin{tikzpicture}
            \node[Cercle] at (0, 0)  (000) {};
        \end{tikzpicture}
    }}}
    \subfigure[$(\EnsEq_4, \TamOrd)$]{\makebox[2cm]{\scalebox{.23}{
        \begin{tikzpicture}
            \node[Cercle] at (0, 0) (1) {};
            \node[Cercle] at (1, 0) (2) {};
            \node[Cercle] at (0, 1) (3) {};
            \node[Cercle] at (1, 1) (4) {};
            \draw[ArreteCube] (1) -- (3);
            \draw[ArreteCube] (1) -- (4);
            \draw[ArreteCube] (2) -- (4);
        \end{tikzpicture}
    }}}
    \subfigure[$(\EnsEq_5, \TamOrd)$]{\makebox[2cm]{\scalebox{.23}{
        \begin{tikzpicture}
            \node[Cercle] at (0, 0)  (1) {};
            \node[Cercle] at (-1, 1) (2) {};
            \node[Cercle] at (1, 1)  (3) {};
            \node[Cercle] at (0, 2)  (4) {};
            \draw[ArreteCube] (1) -- (2);
            \draw[ArreteCube] (1) -- (3);
            \draw[ArreteCube] (2) -- (4);
            \draw[ArreteCube] (3) -- (4);
            \node[Cercle] at (2, 0)  (000) {};
            \node[Cercle] at (2, 1)  (010) {};
            \draw[ArreteCube] (000) -- (010);
        \end{tikzpicture}
    }}}
    \subfigure[$(\EnsEq_6, \TamOrd)$]{\makebox[2cm]{\scalebox{.23}{
        \begin{tikzpicture}
            \node[Cercle] at (0, 0)  (000) {};
            \node[Cercle] at (0, 1)  (010) {};
            \draw[ArreteCube] (000) -- (010);
            \node[Cercle] at (1, 0)  (1) {};
            \node[Cercle] at (1, 1)  (2) {};
            \draw[ArreteCube] (1) -- (2);
        \end{tikzpicture}
    }}}
    \subfigure[$(\EnsEq_7, \TamOrd)$]{\makebox[4cm]{\scalebox{.23}{
        \begin{tikzpicture}
            \node[Cercle] at (0, 0)  (000) {};
            \node[Cercle] at (-1, 1) (100) {};
            \node[Cercle] at (0, 1)  (010) {};
            \node[Cercle] at (1, 1)  (001) {};
            \node[Cercle] at (-1, 2) (110) {};
            \node[Cercle] at (1, 2)  (011) {};
            \node[Cercle] at (0, 2)  (101) {};
            \node[Cercle] at (0, 3)  (111) {};
            \draw[ArreteCube] (000) -- (100);
            \draw[ArreteCube] (000) -- (010);
            \draw[ArreteCube] (000) -- (001);
            \draw[ArreteCube] (100) -- (110);
            \draw[ArreteCube] (010) -- (110);
            \draw[ArreteCube] (010) -- (011);
            \draw[ArreteCube] (001) -- (011);
            \draw[ArreteCube] (100) -- (101);
            \draw[ArreteCube] (001) -- (101);
            \draw[ArreteCube] (101) -- (111);
            \draw[ArreteCube] (110) -- (111);
            \draw[ArreteCube] (011) -- (111);
            \node[Cercle] at (-2.5, 1.5)  (1) {};
            \node[Cercle] at (-3.5, 2.5) (2) {};
            \node[Cercle] at (-1.5, 2.5)  (3) {};
            \node[Cercle] at (-2.5, 3.5)  (4) {};
            \draw[ArreteCube] (1) -- (2);
            \draw[ArreteCube] (1) -- (3);
            \draw[ArreteCube] (2) -- (4);
            \draw[ArreteCube] (3) -- (4);
            \draw[ArreteCube] (1) -- (110);
            \draw[ArreteCube] (3) -- (111);
            \node[Cercle] at (2.5, -0.5)  (1a) {};
            \node[Cercle] at (1.5, 0.5) (2a) {};
            \node[Cercle] at (3.5, 0.5)  (3a) {};
            \node[Cercle] at (2.5, 1.5)  (4a) {};
            \draw[ArreteCube] (1a) -- (2a);
            \draw[ArreteCube] (1a) -- (3a);
            \draw[ArreteCube] (2a) -- (4a);
            \draw[ArreteCube] (3a) -- (4a);
            \draw[ArreteCube] (000) -- (2a);
            \draw[ArreteCube] (001) -- (4a);
            \node[Cercle] at (-2.5, 0)  (b) {};
        \end{tikzpicture}
    }}}
    \subfigure[$(\EnsEq_8, \TamOrd)$]{\makebox[4cm]{\scalebox{.23}{
        \begin{tikzpicture}
            \node[Cercle] at (4, 0)  (000) {};
            \node[Cercle] at (3, 1) (100) {};
            \node[Cercle] at (4, 1)  (010) {};
            \node[Cercle] at (5, 1)  (001) {};
            \node[Cercle] at (3, 2) (110) {};
            \node[Cercle] at (5, 2)  (011) {};
            \node[Cercle] at (4, 2)  (101) {};
            \node[Cercle] at (4, 3)  (111) {};
            \draw[ArreteCube] (000) -- (100);
            \draw[ArreteCube] (000) -- (010);
            \draw[ArreteCube] (000) -- (001);
            \draw[ArreteCube] (100) -- (110);
            \draw[ArreteCube] (010) -- (110);
            \draw[ArreteCube] (010) -- (011);
            \draw[ArreteCube] (001) -- (011);
            \draw[ArreteCube] (100) -- (101);
            \draw[ArreteCube] (001) -- (101);
            \draw[ArreteCube] (101) -- (111);
            \draw[ArreteCube] (110) -- (111);
            \draw[ArreteCube] (011) -- (111);
            \node[Cercle] at (7, 1)  (000b) {};
            \node[Cercle] at (6, 2) (100b) {};
            \node[Cercle] at (7, 2)  (010b) {};
            \node[Cercle] at (8, 2)  (001b) {};
            \node[Cercle] at (6, 3) (110b) {};
            \node[Cercle] at (8, 3)  (011b) {};
            \node[Cercle] at (7, 3)  (101b) {};
            \node[Cercle] at (7, 4)  (111b) {};
            \draw[ArreteCube] (000b) -- (100b);
            \draw[ArreteCube] (000b) -- (010b);
            \draw[ArreteCube] (000b) -- (001b);
            \draw[ArreteCube] (100b) -- (110b);
            \draw[ArreteCube] (010b) -- (110b);
            \draw[ArreteCube] (010b) -- (011b);
            \draw[ArreteCube] (001b) -- (011b);
            \draw[ArreteCube] (100b) -- (101b);
            \draw[ArreteCube] (001b) -- (101b);
            \draw[ArreteCube] (101b) -- (111b);
            \draw[ArreteCube] (110b) -- (111b);
            \draw[ArreteCube] (011b) -- (111b);
            \draw[ArreteCube] (100) -- (100b);
            \draw[ArreteCube] (010) -- (010b);
            \draw[ArreteCube] (001) -- (001b);
            \draw[ArreteCube] (110) -- (110b);
            \draw[ArreteCube] (011) -- (011b);
            \draw[ArreteCube] (101) -- (101b);
            \draw[ArreteCube] (111) -- (111b);
            \draw[ArreteCube] (000) -- (000b);
            \node[Cercle] at (0, 0)  (0) {};
            \node[Cercle] at (1, 1)  (1) {};
            \node[Cercle] at (2, 0)  (2) {};
            \node[Cercle] at (0, 2)  (3) {};
            \node[Cercle] at (2, 2)  (4) {};
            \node[Cercle] at (1, 3)  (5) {};
            \node[Cercle] at (-1, 1)  (6) {};
            \node[Cercle] at (-2, 0)  (7) {};
            \draw[ArreteCube] (0) -- (1);
            \draw[ArreteCube] (2) -- (1);
            \draw[ArreteCube] (0) -- (3);
            \draw[ArreteCube] (2) -- (4);
            \draw[ArreteCube] (1) -- (5);
            \draw[ArreteCube] (3) -- (5);
            \draw[ArreteCube] (4) -- (5);
            \draw[ArreteCube] (0) -- (6);
            \draw[ArreteCube] (7) -- (6);
            \node[Cercle] at (11, 3)  (0a) {};
            \node[Cercle] at (10, 2) (1a) {};
            \node[Cercle] at (9, 3) (2a) {};
            \node[Cercle] at (11, 1)  (3a) {};
            \node[Cercle] at (9, 1) (4a) {};
            \node[Cercle] at (10, 0) (5a) {};
            \node[Cercle] at (12, 2)  (6a) {};
            \node[Cercle] at (13, 3)  (7a) {};
            \draw[ArreteCube] (0a) -- (1a);
            \draw[ArreteCube] (2a) -- (1a);
            \draw[ArreteCube] (0a) -- (3a);
            \draw[ArreteCube] (2a) -- (4a);
            \draw[ArreteCube] (1a) -- (5a);
            \draw[ArreteCube] (3a) -- (5a);
            \draw[ArreteCube] (4a) -- (5a);
            \draw[ArreteCube] (0a) -- (6a);
            \draw[ArreteCube] (7a) -- (6a);
        \end{tikzpicture}
    }}}
    \subfigure[$(\EnsEq_9, \TamOrd)$]{\makebox[4cm]{\scalebox{.23}{
        \begin{tikzpicture}
            \node[Cercle] at (0, 0)  (1) {};
            \node[Cercle] at (-1, 1) (2) {};
            \node[Cercle] at (1, 1)  (3) {};
            \node[Cercle] at (0, 2)  (4) {};
            \draw[ArreteCube] (1) -- (2);
            \draw[ArreteCube] (1) -- (3);
            \draw[ArreteCube] (2) -- (4);
            \draw[ArreteCube] (3) -- (4);
            \node[Cercle] at (-3, 0) (1a) {};
            \node[Cercle] at (-4, 1) (2a) {};
            \node[Cercle] at (-2, 1) (3a) {};
            \node[Cercle] at (-3, 2) (4a) {};
            \draw[ArreteCube] (1a) -- (2a);
            \draw[ArreteCube] (1a) -- (3a);
            \draw[ArreteCube] (2a) -- (4a);
            \draw[ArreteCube] (3a) -- (4a);
            \node[Cercle] at (3, 0) (1b) {};
            \node[Cercle] at (2, 1) (2b) {};
            \node[Cercle] at (4, 1) (3b) {};
            \node[Cercle] at (3, 2) (4b) {};
            \draw[ArreteCube] (1b) -- (2b);
            \draw[ArreteCube] (1b) -- (3b);
            \draw[ArreteCube] (2b) -- (4b);
            \draw[ArreteCube] (3b) -- (4b);
            \draw[ArreteCube] (1) -- (3a);
            \draw[ArreteCube] (2) -- (4a);
            \draw[ArreteCube] (1) -- (2b);
            \draw[ArreteCube] (3) -- (4b);
            \node[Cercle] at (0, 3)  (1c) {};
            \node[Cercle] at (-1, 4) (2c) {};
            \node[Cercle] at (1, 4)  (3c) {};
            \node[Cercle] at (0, 5)  (4c) {};
            \draw[ArreteCube] (1c) -- (2c);
            \draw[ArreteCube] (1c) -- (3c);
            \draw[ArreteCube] (2c) -- (4c);
            \draw[ArreteCube] (3c) -- (4c);
            \draw[ArreteCube] (1a) -- (2c);
            \draw[ArreteCube] (3a) -- (4c);
            \draw[ArreteCube] (1b) -- (3c);
            \draw[ArreteCube] (2b) -- (4c);
            \node[Cercle] at (0, -4)  (000a) {};
            \node[Cercle] at (-1, -3) (100a) {};
            \node[Cercle] at (0, -3)  (010a) {};
            \node[Cercle] at (1, -3)  (001a) {};
            \node[Cercle] at (-1, -2) (110a) {};
            \node[Cercle] at (1, -2)  (011a) {};
            \node[Cercle] at (0, -2)  (101a) {};
            \node[Cercle] at (0, -1)  (111a) {};
            \draw[ArreteCube] (000a) -- (100a);
            \draw[ArreteCube] (000a) -- (010a);
            \draw[ArreteCube] (000a) -- (001a);
            \draw[ArreteCube] (100a) -- (110a);
            \draw[ArreteCube] (010a) -- (110a);
            \draw[ArreteCube] (010a) -- (011a);
            \draw[ArreteCube] (001a) -- (011a);
            \draw[ArreteCube] (100a) -- (101a);
            \draw[ArreteCube] (001a) -- (101a);
            \draw[ArreteCube] (101a) -- (111a);
            \draw[ArreteCube] (110a) -- (111a);
            \draw[ArreteCube] (011a) -- (111a);
            \node[Cercle] at (-3, -4)  (000b) {};
            \node[Cercle] at (-4, -3) (100b) {};
            \node[Cercle] at (-3, -3)  (010b) {};
            \node[Cercle] at (-2, -3)  (001b) {};
            \node[Cercle] at (-4, -2) (110b) {};
            \node[Cercle] at (-2, -2)  (011b) {};
            \node[Cercle] at (-3, -2)  (101b) {};
            \node[Cercle] at (-3, -1)  (111b) {};
            \draw[ArreteCube] (000b) -- (100b);
            \draw[ArreteCube] (000b) -- (010b);
            \draw[ArreteCube] (000b) -- (001b);
            \draw[ArreteCube] (100b) -- (110b);
            \draw[ArreteCube] (010b) -- (110b);
            \draw[ArreteCube] (010b) -- (011b);
            \draw[ArreteCube] (001b) -- (011b);
            \draw[ArreteCube] (100b) -- (101b);
            \draw[ArreteCube] (001b) -- (101b);
            \draw[ArreteCube] (101b) -- (111b);
            \draw[ArreteCube] (110b) -- (111b);
            \draw[ArreteCube] (011b) -- (111b);
            \node[Cercle] at (3, -4)  (000c) {};
            \node[Cercle] at (2, -3) (100c) {};
            \node[Cercle] at (3, -3)  (010c) {};
            \node[Cercle] at (4, -3)  (001c) {};
            \node[Cercle] at (2, -2) (110c) {};
            \node[Cercle] at (4, -2)  (011c) {};
            \node[Cercle] at (3, -2)  (101c) {};
            \node[Cercle] at (3, -1)  (111c) {};
            \draw[ArreteCube] (000c) -- (100c);
            \draw[ArreteCube] (000c) -- (010c);
            \draw[ArreteCube] (000c) -- (001c);
            \draw[ArreteCube] (100c) -- (110c);
            \draw[ArreteCube] (010c) -- (110c);
            \draw[ArreteCube] (010c) -- (011c);
            \draw[ArreteCube] (001c) -- (011c);
            \draw[ArreteCube] (100c) -- (101c);
            \draw[ArreteCube] (001c) -- (101c);
            \draw[ArreteCube] (101c) -- (111c);
            \draw[ArreteCube] (110c) -- (111c);
            \draw[ArreteCube] (011c) -- (111c);
            \node[Cercle] at (5, 0) (1d) {};
            \node[Cercle] at (6, 0) (2d) {};
            \node[Cercle] at (5, 1) (3d) {};
            \node[Cercle] at (6, 1) (4d) {};
            \draw[ArreteCube] (1d) -- (3d);
            \draw[ArreteCube] (1d) -- (4d);
            \draw[ArreteCube] (2d) -- (4d);
        \end{tikzpicture}
    }}}
    \subfigure[$(\EnsEq_{10}, \TamOrd)$]{\makebox[4cm]{\scalebox{.23}{
        \begin{tikzpicture}
            \node[Cercle] at (0, 0)  (000a) {};
            \node[Cercle] at (-1, 1) (100a) {};
            \node[Cercle] at (0, 1)  (010a) {};
            \node[Cercle] at (1, 1)  (001a) {};
            \node[Cercle] at (-1, 2) (110a) {};
            \node[Cercle] at (1, 2)  (011a) {};
            \node[Cercle] at (0, 2)  (101a) {};
            \node[Cercle] at (0, 3)  (111a) {};
            \draw[ArreteCube] (000a) -- (100a);
            \draw[ArreteCube] (000a) -- (010a);
            \draw[ArreteCube] (000a) -- (001a);
            \draw[ArreteCube] (100a) -- (110a);
            \draw[ArreteCube] (010a) -- (110a);
            \draw[ArreteCube] (010a) -- (011a);
            \draw[ArreteCube] (001a) -- (011a);
            \draw[ArreteCube] (100a) -- (101a);
            \draw[ArreteCube] (001a) -- (101a);
            \draw[ArreteCube] (101a) -- (111a);
            \draw[ArreteCube] (110a) -- (111a);
            \draw[ArreteCube] (011a) -- (111a);
            \node[Cercle] at (3, 0.5)  (000b) {};
            \node[Cercle] at (2, 1.5) (100b) {};
            \node[Cercle] at (3, 1.5)  (010b) {};
            \node[Cercle] at (4, 1.5)  (001b) {};
            \node[Cercle] at (2, 2.5) (110b) {};
            \node[Cercle] at (4, 2.5)  (011b) {};
            \node[Cercle] at (3, 2.5)  (101b) {};
            \node[Cercle] at (3, 3.5)  (111b) {};
            \draw[ArreteCube] (000b) -- (100b);
            \draw[ArreteCube] (000b) -- (010b);
            \draw[ArreteCube] (000b) -- (001b);
            \draw[ArreteCube] (100b) -- (110b);
            \draw[ArreteCube] (010b) -- (110b);
            \draw[ArreteCube] (010b) -- (011b);
            \draw[ArreteCube] (001b) -- (011b);
            \draw[ArreteCube] (100b) -- (101b);
            \draw[ArreteCube] (001b) -- (101b);
            \draw[ArreteCube] (101b) -- (111b);
            \draw[ArreteCube] (110b) -- (111b);
            \draw[ArreteCube] (011b) -- (111b);
            \draw[ArreteCube] (001a) -- (000b);
            \draw[ArreteCube] (111a) -- (110b);
            \draw[ArreteCube] (010b) -- (011a);
            \draw[ArreteCube] (100b) -- (101a);
            \node[Cercle] at (6, 0)  (000a2) {};
            \node[Cercle] at (5, 1) (100a2) {};
            \node[Cercle] at (6, 1)  (010a2) {};
            \node[Cercle] at (7, 1)  (001a2) {};
            \node[Cercle] at (5, 2) (110a2) {};
            \node[Cercle] at (7, 2)  (011a2) {};
            \node[Cercle] at (6, 2)  (101a2) {};
            \node[Cercle] at (6, 3)  (111a2) {};
            \draw[ArreteCube] (000a2) -- (100a2);
            \draw[ArreteCube] (000a2) -- (010a2);
            \draw[ArreteCube] (000a2) -- (001a2);
            \draw[ArreteCube] (100a2) -- (110a2);
            \draw[ArreteCube] (010a2) -- (110a2);
            \draw[ArreteCube] (010a2) -- (011a2);
            \draw[ArreteCube] (001a2) -- (011a2);
            \draw[ArreteCube] (100a2) -- (101a2);
            \draw[ArreteCube] (001a2) -- (101a2);
            \draw[ArreteCube] (101a2) -- (111a2);
            \draw[ArreteCube] (110a2) -- (111a2);
            \draw[ArreteCube] (011a2) -- (111a2);
            \node[Cercle] at (9, 0.5)  (000b2) {};
            \node[Cercle] at (8, 1.5) (100b2) {};
            \node[Cercle] at (9, 1.5)  (010b2) {};
            \node[Cercle] at (10, 1.5)  (001b2) {};
            \node[Cercle] at (8, 2.5) (110b2) {};
            \node[Cercle] at (10, 2.5)  (011b2) {};
            \node[Cercle] at (9, 2.5)  (101b2) {};
            \node[Cercle] at (9, 3.5)  (111b2) {};
            \draw[ArreteCube] (000b2) -- (100b2);
            \draw[ArreteCube] (000b2) -- (010b2);
            \draw[ArreteCube] (000b2) -- (001b2);
            \draw[ArreteCube] (100b2) -- (110b2);
            \draw[ArreteCube] (010b2) -- (110b2);
            \draw[ArreteCube] (010b2) -- (011b2);
            \draw[ArreteCube] (001b2) -- (011b2);
            \draw[ArreteCube] (100b2) -- (101b2);
            \draw[ArreteCube] (001b2) -- (101b2);
            \draw[ArreteCube] (101b2) -- (111b2);
            \draw[ArreteCube] (110b2) -- (111b2);
            \draw[ArreteCube] (011b2) -- (111b2);
            \draw[ArreteCube] (001a2) -- (000b2);
            \draw[ArreteCube] (111a2) -- (110b2);
            \draw[ArreteCube] (010b2) -- (011a2);
            \draw[ArreteCube] (100b2) -- (101a2);
            \node[Cercle] at (0, 4)  (0) {};
            \node[Cercle] at (1, 5)  (1) {};
            \node[Cercle] at (2, 4)  (2) {};
            \node[Cercle] at (0, 6)  (3) {};
            \node[Cercle] at (2, 6)  (4) {};
            \node[Cercle] at (1, 7)  (5) {};
            \node[Cercle] at (3, 5)  (6) {};
            \node[Cercle] at (3, 7)  (7) {};
            \draw[ArreteCube] (0) -- (1);
            \draw[ArreteCube] (2) -- (1);
            \draw[ArreteCube] (0) -- (3);
            \draw[ArreteCube] (2) -- (4);
            \draw[ArreteCube] (1) -- (5);
            \draw[ArreteCube] (3) -- (5);
            \draw[ArreteCube] (4) -- (5);
            \draw[ArreteCube] (6) -- (7);
            \draw[ArreteCube] (2) -- (6);
            \draw[ArreteCube] (4) -- (7);
            \node[Cercle] at (5, 4)  (0a) {};
            \node[Cercle] at (6, 5)  (1a) {};
            \node[Cercle] at (7, 4)  (2a) {};
            \node[Cercle] at (5, 6)  (3a) {};
            \node[Cercle] at (7, 6)  (4a) {};
            \node[Cercle] at (6, 7)  (5a) {};
            \node[Cercle] at (8, 5)  (6a) {};
            \node[Cercle] at (8, 7)  (7a) {};
            \draw[ArreteCube] (0a) -- (1a);
            \draw[ArreteCube] (2a) -- (1a);
            \draw[ArreteCube] (0a) -- (3a);
            \draw[ArreteCube] (2a) -- (4a);
            \draw[ArreteCube] (1a) -- (5a);
            \draw[ArreteCube] (3a) -- (5a);
            \draw[ArreteCube] (4a) -- (5a);
            \draw[ArreteCube] (6a) -- (7a);
            \draw[ArreteCube] (2a) -- (6a);
            \draw[ArreteCube] (4a) -- (7a);
            \node[Cercle] at (1, 8)  (1b) {};
            \node[Cercle] at (0, 9) (2b) {};
            \node[Cercle] at (2, 9)  (3b) {};
            \node[Cercle] at (1, 10)  (4b) {};
            \draw[ArreteCube] (1b) -- (2b);
            \draw[ArreteCube] (1b) -- (3b);
            \draw[ArreteCube] (2b) -- (4b);
            \draw[ArreteCube] (3b) -- (4b);
            \node[Cercle] at (4, 8)  (1c) {};
            \node[Cercle] at (3, 9) (2c) {};
            \node[Cercle] at (5, 9)  (3c) {};
            \node[Cercle] at (4, 10)  (4c) {};
            \draw[ArreteCube] (1c) -- (2c);
            \draw[ArreteCube] (1c) -- (3c);
            \draw[ArreteCube] (2c) -- (4c);
            \draw[ArreteCube] (3c) -- (4c);
            \node[Cercle] at (7, 8)  (1d) {};
            \node[Cercle] at (6, 9) (2d) {};
            \node[Cercle] at (8, 9)  (3d) {};
            \node[Cercle] at (7, 10)  (4d) {};
            \draw[ArreteCube] (1d) -- (2d);
            \draw[ArreteCube] (1d) -- (3d);
            \draw[ArreteCube] (2d) -- (4d);
            \draw[ArreteCube] (3d) -- (4d);
        \end{tikzpicture}
    }}}
    \subfigure[$(\EnsEq_{11}, \TamOrd)$]{\makebox[4cm]{\scalebox{.23}{
        \begin{tikzpicture}
            \node[Cercle] at (4, 0)  (000) {};
            \node[Cercle] at (3, 1) (100) {};
            \node[Cercle] at (4, 1)  (010) {};
            \node[Cercle] at (5, 1)  (001) {};
            \node[Cercle] at (3, 2) (110) {};
            \node[Cercle] at (5, 2)  (011) {};
            \node[Cercle] at (4, 2)  (101) {};
            \node[Cercle] at (4, 3)  (111) {};
            \draw[ArreteCube] (000) -- (100);
            \draw[ArreteCube] (000) -- (010);
            \draw[ArreteCube] (000) -- (001);
            \draw[ArreteCube] (100) -- (110);
            \draw[ArreteCube] (010) -- (110);
            \draw[ArreteCube] (010) -- (011);
            \draw[ArreteCube] (001) -- (011);
            \draw[ArreteCube] (100) -- (101);
            \draw[ArreteCube] (001) -- (101);
            \draw[ArreteCube] (101) -- (111);
            \draw[ArreteCube] (110) -- (111);
            \draw[ArreteCube] (011) -- (111);
            \node[Cercle] at (7, 1)  (000b) {};
            \node[Cercle] at (6, 2) (100b) {};
            \node[Cercle] at (7, 2)  (010b) {};
            \node[Cercle] at (8, 2)  (001b) {};
            \node[Cercle] at (6, 3) (110b) {};
            \node[Cercle] at (8, 3)  (011b) {};
            \node[Cercle] at (7, 3)  (101b) {};
            \node[Cercle] at (7, 4)  (111b) {};
            \draw[ArreteCube] (000b) -- (100b);
            \draw[ArreteCube] (000b) -- (010b);
            \draw[ArreteCube] (000b) -- (001b);
            \draw[ArreteCube] (100b) -- (110b);
            \draw[ArreteCube] (010b) -- (110b);
            \draw[ArreteCube] (010b) -- (011b);
            \draw[ArreteCube] (001b) -- (011b);
            \draw[ArreteCube] (100b) -- (101b);
            \draw[ArreteCube] (001b) -- (101b);
            \draw[ArreteCube] (101b) -- (111b);
            \draw[ArreteCube] (110b) -- (111b);
            \draw[ArreteCube] (011b) -- (111b);
            \draw[ArreteCube] (100) -- (100b);
            \draw[ArreteCube] (010) -- (010b);
            \draw[ArreteCube] (001) -- (001b);
            \draw[ArreteCube] (110) -- (110b);
            \draw[ArreteCube] (011) -- (011b);
            \draw[ArreteCube] (101) -- (101b);
            \draw[ArreteCube] (111) -- (111b);
            \draw[ArreteCube] (000) -- (000b);
            \node[Cercle] at (1, 0)  (000c) {};
            \node[Cercle] at (0, 1) (100c) {};
            \node[Cercle] at (1, 1)  (010c) {};
            \node[Cercle] at (2, 1)  (001c) {};
            \node[Cercle] at (0, 2) (110c) {};
            \node[Cercle] at (2, 2)  (011c) {};
            \node[Cercle] at (1, 2)  (101c) {};
            \node[Cercle] at (1, 3)  (111c) {};
            \draw[ArreteCube] (000c) -- (100c);
            \draw[ArreteCube] (000c) -- (010c);
            \draw[ArreteCube] (000c) -- (001c);
            \draw[ArreteCube] (100c) -- (110c);
            \draw[ArreteCube] (010c) -- (110c);
            \draw[ArreteCube] (010c) -- (011c);
            \draw[ArreteCube] (001c) -- (011c);
            \draw[ArreteCube] (100c) -- (101c);
            \draw[ArreteCube] (001c) -- (101c);
            \draw[ArreteCube] (101c) -- (111c);
            \draw[ArreteCube] (110c) -- (111c);
            \draw[ArreteCube] (011c) -- (111c);
            \node[Cercle] at (10, 0)  (000d) {};
            \node[Cercle] at (9, 1) (100d) {};
            \node[Cercle] at (10, 1)  (010d) {};
            \node[Cercle] at (11, 1)  (001d) {};
            \node[Cercle] at (9, 2) (110d) {};
            \node[Cercle] at (11, 2)  (011d) {};
            \node[Cercle] at (10, 2)  (101d) {};
            \node[Cercle] at (10, 3)  (111d) {};
            \draw[ArreteCube] (000d) -- (100d);
            \draw[ArreteCube] (000d) -- (010d);
            \draw[ArreteCube] (000d) -- (001d);
            \draw[ArreteCube] (100d) -- (110d);
            \draw[ArreteCube] (010d) -- (110d);
            \draw[ArreteCube] (010d) -- (011d);
            \draw[ArreteCube] (001d) -- (011d);
            \draw[ArreteCube] (100d) -- (101d);
            \draw[ArreteCube] (001d) -- (101d);
            \draw[ArreteCube] (101d) -- (111d);
            \draw[ArreteCube] (110d) -- (111d);
            \draw[ArreteCube] (011d) -- (111d);
            \node[Cercle] at (0, 4)  (0e) {};
            \node[Cercle] at (1, 5)  (1e) {};
            \node[Cercle] at (2, 4)  (2e) {};
            \node[Cercle] at (0, 6)  (3e) {};
            \node[Cercle] at (2, 6)  (4e) {};
            \node[Cercle] at (1, 7)  (5e) {};
            \node[Cercle] at (3, 5)  (6e) {};
            \node[Cercle] at (3, 7)  (7e) {};
            \draw[ArreteCube] (0e) -- (1e);
            \draw[ArreteCube] (2e) -- (1e);
            \draw[ArreteCube] (0e) -- (3e);
            \draw[ArreteCube] (2e) -- (4e);
            \draw[ArreteCube] (1e) -- (5e);
            \draw[ArreteCube] (3e) -- (5e);
            \draw[ArreteCube] (4e) -- (5e);
            \draw[ArreteCube] (6e) -- (7e);
            \draw[ArreteCube] (2e) -- (6e);
            \draw[ArreteCube] (4e) -- (7e);
            \node[Cercle] at (4, 4)  (0f) {};
            \node[Cercle] at (5, 5)  (1f) {};
            \node[Cercle] at (6, 4)  (2f) {};
            \node[Cercle] at (4, 6)  (3f) {};
            \node[Cercle] at (6, 6)  (4f) {};
            \node[Cercle] at (5, 7)  (5f) {};
            \node[Cercle] at (7, 5)  (6f) {};
            \node[Cercle] at (7, 7)  (7f) {};
            \draw[ArreteCube] (0f) -- (1f);
            \draw[ArreteCube] (2f) -- (1f);
            \draw[ArreteCube] (0f) -- (3f);
            \draw[ArreteCube] (2f) -- (4f);
            \draw[ArreteCube] (1f) -- (5f);
            \draw[ArreteCube] (3f) -- (5f);
            \draw[ArreteCube] (4f) -- (5f);
            \draw[ArreteCube] (6f) -- (7f);
            \draw[ArreteCube] (2f) -- (6f);
            \draw[ArreteCube] (4f) -- (7f);
            \node[Cercle] at (8, 4)  (0g) {};
            \node[Cercle] at (9, 5)  (1g) {};
            \node[Cercle] at (10, 4)  (2g) {};
            \node[Cercle] at (8, 6)  (3g) {};
            \node[Cercle] at (10, 6)  (4g) {};
            \node[Cercle] at (9, 7)  (5g) {};
            \node[Cercle] at (11, 5)  (6g) {};
            \node[Cercle] at (11, 7)  (7g) {};
            \draw[ArreteCube] (0g) -- (1g);
            \draw[ArreteCube] (2g) -- (1g);
            \draw[ArreteCube] (0g) -- (3g);
            \draw[ArreteCube] (2g) -- (4g);
            \draw[ArreteCube] (1g) -- (5g);
            \draw[ArreteCube] (3g) -- (5g);
            \draw[ArreteCube] (4g) -- (5g);
            \draw[ArreteCube] (6g) -- (7g);
            \draw[ArreteCube] (2g) -- (6g);
            \draw[ArreteCube] (4g) -- (7g);
            \node[Cercle] at (0, 7)  (0h) {};
            \node[Cercle] at (1, 8)  (1h) {};
            \node[Cercle] at (2, 7)  (2h) {};
            \node[Cercle] at (0, 9)  (3h) {};
            \node[Cercle] at (2, 9)  (4h) {};
            \node[Cercle] at (1, 10)  (5h) {};
            \node[Cercle] at (3, 8)  (6h) {};
            \node[Cercle] at (3, 10)  (7h) {};
            \draw[ArreteCube] (0h) -- (1h);
            \draw[ArreteCube] (2h) -- (1h);
            \draw[ArreteCube] (0h) -- (3h);
            \draw[ArreteCube] (2h) -- (4h);
            \draw[ArreteCube] (1h) -- (5h);
            \draw[ArreteCube] (3h) -- (5h);
            \draw[ArreteCube] (4h) -- (5h);
            \draw[ArreteCube] (6h) -- (7h);
            \draw[ArreteCube] (2h) -- (6h);
            \draw[ArreteCube] (4h) -- (7h);
            \node[Cercle] at (6, 7)  (1i) {};
            \node[Cercle] at (5, 8) (2i) {};
            \node[Cercle] at (7, 8)  (3i) {};
            \node[Cercle] at (6, 9)  (4i) {};
            \draw[ArreteCube] (1i) -- (2i);
            \draw[ArreteCube] (1i) -- (3i);
            \draw[ArreteCube] (2i) -- (4i);
            \draw[ArreteCube] (3i) -- (4i);
            \node[Cercle] at (10, 7)  (000j) {};
            \node[Cercle] at (10, 8)  (010j) {};
            \draw[ArreteCube] (000j) -- (010j);
        \end{tikzpicture}
    }}}
    \caption{Hasse diagrams of the first $(\EnsEq_n, \TamOrd)$ posets.}
    \label{FIGdiagInterEq}
\end{figure}

\subsection{Enumeration of balanced tree intervals}

Let us make use again of the synchronous grammars:

\begin{proposition} \label{PROserieGenIntEq}
    The generating series enumerating balanced tree intervals in the
    Tamari lattice according to the number of leaves of the trees is
    $G_{\textnormal{inter}}(x) := A(x, 0, 0)$ where
    \begin{equation}
        A(x, y, z) := x + A(x^2 + 2xy + z, x, x^3 + x^2y).
    \end{equation}
\end{proposition}
\begin{proof}
    Let $I = [T_0, T_1]$ be a balanced tree interval. This interval can be
    encoded by the tree $T_0$ in which we mark the nodes which are roots
    of the conservative balancing rotations needed to transform $T_0$ into
    $T_1$. If a node $y$ of $T_0$ is marked, then its left child cannot be
    marked too because the rotations of the interval $I$ are disjoint (see
    the proof of Theorem \ref{THEintervalleHypercube}). To generate these
    objects, we use the following synchronous grammar that generates marked
    trees (the marked nodes are represented by a rectangle instead of a circle):
    \begin{eqnarray}
        \raisebox{0.65em}{\Bourgeon{.45}{$x$}} & \raisebox{0.75em}{$\longrightarrow$} &
                \BourgeonA{.45}{$x$}{-1}{$y$} ~\raisebox{0.75em}{$+$}~
                \BourgeonA{.45}{$x$}{$0$}{$x$} ~\raisebox{0.75em}{$+$}~
                \BourgeonA{.45}{$y$}{$1$}{$x$} ~\raisebox{0.75em}{$+$}~
                \raisebox{0.75em}{\Bourgeon{.45}{$z$}} \\
        \Bourgeon{.45}{$y$} & \raisebox{0.10em}{$\longrightarrow$} & \Bourgeon{.45}{$x$} \\
        \raisebox{1.1em}{\Bourgeon{.45}{$z$}} & \raisebox{1.2em}{$\longrightarrow$} &
                \BourgeonB{.45}{$0$}{$x$}{-1} ~~~\raisebox{1.2em}{$+$}~
                \BourgeonB{.45}{-1}{$y$}{-1}
    \end{eqnarray}
\end{proof}

The solution of this functional equation gives us the following first values
for the number of balanced tree intervals in the Tamari lattice:
$1$, $1$, $3$, $1$, $7$, $12$, $6$, $52$, $119$, $137$, $195$,
$231$, $1019$, $3503$, $6593$, $12616$, $26178$, $43500$, $64157$,
$94688$, $232560$, $817757$, $2233757$, $5179734$.

The interval $[T_0, T_1]$ is a \emph{maximal balanced tree interval}
if $T_0$ (resp. $T_1$) is a minimal (resp. maximal) balanced tree.

\begin{proposition} \label{PROserieGenIntArbresEqMax}
    The generating series enumerating maximal balanced tree intervals in
    the Tamari lattice according to the number of leaves of the trees is
    $G_{\textnormal{intermax}}(x) := A(x, 0, 0, 0)$ where
    \begin{equation}
        A(x, y, z, t) := x + A(x^2 + 2yz + t, x, yz + t, x^3 + x^2y).
    \end{equation}
\end{proposition}
\begin{proof}
    Let $I = [T_0, T_1]$ be a maximal balanced tree interval. This interval
    can be encoded by the minimal tree $T_0$ in which we mark the nodes
    which are roots of the conservative balancing rotations needed to
    transform $T_0$ into $T_1$. Since $T_1$ is a maximal balanced tree,
    by Proposition \ref{PRO1}, it avoids the tree patterns of $P_{\textnormal{max}}$,
    thus, the object which encodes $I$ must not have a node which is root
    of a conservative balancing rotation not marked if its parent or its
    left child is not marked. To generate these objects, we use the
    following synchronous grammar:
    \begin{eqnarray}
        \raisebox{0.65em}{\Bourgeon{.45}{$x$}} & \raisebox{0.75em}{$\longrightarrow$} &
                \BourgeonA{.45}{$x$}{$0$}{$x$} ~\raisebox{0.75em}{$+$}~
                \BourgeonA{.45}{$y$}{$1$}{$z_1$} ~\raisebox{0.75em}{$+$}~
                \BourgeonA{.45}{$z_2$}{-1}{$y$} ~\raisebox{0.75em}{$+$}~
                \raisebox{0.65em}{\Bourgeon{.45}{$t$}} \\
        \Bourgeon{.45}{$y$} & \raisebox{0.10em}{$\longrightarrow$} & \Bourgeon{.45}{$x$} \\
        \raisebox{0.65em}{\Bourgeon{.45}{$z_1$}} & \raisebox{0.75em}{$\longrightarrow$} &
                \BourgeonA{.45}{$z_2$}{-1}{$y$} ~\raisebox{0.75em}{$+$}~
                \raisebox{0.65em}{\Bourgeon{.45}{$t$}} \\
        \raisebox{0.65em}{\Bourgeon{.45}{$z_2$}} & \raisebox{0.75em}{$\longrightarrow$} &
                \BourgeonA{.45}{$y$}{$1$}{$z_1$} ~\raisebox{0.75em}{$+$}~
                \raisebox{0.65em}{\Bourgeon{.45}{$t$}} \\
        \raisebox{1.1em}{\Bourgeon{.45}{$t$}} & \raisebox{1.2em}{$\longrightarrow$} &
                \BourgeonB{.45}{$0$}{$x$}{-1} ~~~\raisebox{1.2em}{$+$}~
                \BourgeonB{.45}{-1}{$y$}{-1}
    \end{eqnarray}
    Note that the buds \Bourgeon{.45}{$z_1$} and \Bourgeon{.45}{$z_2$}
    play the same role so that the functional equation is simplified.
\end{proof}

The solution of this functional equation gives us the following first values
for the number of maximal balanced tree intervals in the Tamari lattice:
$1$, $1$, $1$, $1$, $3$, $2$, $2$, $6$,
$9$, $15$, $15$, $17$, $41$, $77$, $125$, $178$, $252$, $376$, $531$,
$740$, $1192$, $2179$, $4273$, $7738$, $13012$, $20776$, $32389$, $49841$,
$75457$, $113011$.

\bibliographystyle{plain}
\bibliography{Bibliographie}

\end{document}